\RequirePackage{fix-cm}
\documentclass{article}
\usepackage[T1]{fontenc}
\usepackage[utf8]{inputenc}
\usepackage{verbatim}
\usepackage{subfig}
\usepackage{graphicx}
\usepackage{amssymb} %
\usepackage{amsthm}
\newtheorem{theorem}{Theorem}
\newtheorem{proposition}{Proposition}
\newtheorem{lemma}{Lemma}
\newtheorem{corollary}{Corollary}
\theoremstyle{definition}

\newtheorem{remark}{Remark}
\usepackage{titlesec}
\titleformat{\paragraph}[runin]{\normalsize\itshape}{}{15pt}{}
\titlespacing{\paragraph}{0pt}{24pt}{5pt}
\titleformat{\section}[block]{\normalsize\bfseries\boldmath}{\thesection}{3.5pt}{}
\titleformat{\subsection}[block]{\normalsize\upshape}{\thesubsection}{3.5pt}{}

\makeatletter
\@ifpackagelater{titlesec}{2016/03/21}{}{%
 \@ifpackagelater{titlesec}{2016/03/15}{%
  \usepackage{etoolbox}%
  \makeatletter
  \patchcmd{\ttlh@hang}{\parindent\z@}{\parindent\z@\leavevmode}{}{}%
  \patchcmd{\ttlh@hang}{\noindent}{}{}{}%
  \makeatother
 }{}%
}
\makeatother

\usepackage{apptools}
\AtAppendix{%
    \titleformat{\section}[block]%
        {\normalsize\bfseries\boldmath}%
        {\appendixname~\thesection:}%
        {0.333em}{}%
}
\usepackage{tikz}
\usetikzlibrary{positioning}
\usetikzlibrary{calc}
\usetikzlibrary{spy}
\newcommand*{\ExtractCoordinate}[1]{\path (#1); \pgfgetlastxy{\XCoord}{\YCoord}; }%
\newcommand*{\ExtractImgDims}[1]{
    \ExtractCoordinate{$(#1.south west)$};
    \pgfmathsetmacro{\imgx}{\XCoord}
    \pgfmathsetmacro{\imgy}{\YCoord}
    \ExtractCoordinate{$(#1.north east)$};
    \pgfmathsetmacro{\imgw}{\XCoord - \imgx}
    \pgfmathsetmacro{\imgh}{\YCoord - \imgy}
}
\newcommand*{\RelativeSpy}[4]{
    \ExtractImgDims{#2};
    \begin{scope}[x=\imgw,y=\imgh,xshift=\imgx,yshift=\imgy]
        \coordinate (spyroi-#1) at #3;
        \coordinate (spypos-#1) at #4;
        \spy[anchor=center] on (spyroi-#1) in node[anchor=center] at (spypos-#1);
    \end{scope}
}
\usepackage{mathtools}
\newcommand{\R}{\mathbb{R}}
\newcommand{\IN}{\mathbb{N}}
\newcommand{\IP}{\mathcal{P}}
\newcommand{\IM}{\mathcal{M}}
\newcommand{\IL}{\mathcal{L}}
\newcommand{\IH}{\mathcal{H}}
\newcommand{\IS}{\mathbb{S}}
\newcommand{\nbhd}{\mathcal{N}}
\DeclareMathOperator*{\argmin}{arg\,min}
\newcommand{\esssup}{\operatorname{ess\,sup}}
\newcommand{\supp}{\operatorname{supp}}
\newcommand{\Lip}{\operatorname{Lip}}
\newcommand{\Per}{\operatorname{Per}}

\newcommand{\TV}{\operatorname{TV}}
\newcommand{\SNR}{\operatorname{SNR}}
\newcommand{\KR}{K\!R}
\renewcommand{\div}{\operatorname{div}}
\newcommand{\dd}{\,d}

\newcommand{\proj}{\pi}

\usepackage{soulutf8}
\soulregister\cite7
\soulregister\ref7
\soulregister\pageref7
\soulregister\eqref7

\providecommand{\journalname}[1]{}
\providecommand{\keywords}[1]{\par{\bf Key words:} #1}
\providecommand{\institute}[1]{\date{\normalsize \textit{#1}}}
\journalname{J Math Imaging Vis}
\begin{document}

\title{Measure-Valued Variational Models with Applications %
       to Diffusion-Weighted Imaging}

\author{Thomas Vogt \and Jan Lellmann}

\institute{%
Institute of Mathematics and Image Computing (MIC), University of Lübeck, \\
Maria-Goeppert-Str. 3, 23562 Lübeck, Germany \\
\{vogt,lellmann\}@mic.uni-luebeck.de
}

\maketitle

\begin{abstract}
We develop a general mathematical framework for variational problems where the
unknown function takes values in the space of probability measures on some
metric space.
We study weak and strong topologies and define a total variation seminorm
for functions taking values in a Banach space.
The seminorm penalizes jumps and is rotationally invariant under certain
conditions.
We prove existence of a minimizer for a class of variational problems based on
this formulation of total variation, and provide an example where uniqueness
fails to hold.
Employing the Kan\-torovich-Rubinstein transport norm from the theory of optimal
transport, we propose a variational approach for the restoration of orientation
distribution function (ODF)-valued images, as commonly used in Diffusion MRI.
We demonstrate that the approach is numerically feasible on several data sets.
\end{abstract}

\noindent
\section{Introduction}
In this work, we are concerned with variational problems in which the unknown
function $u\colon \Omega \to \IP(\IS^2)$ maps from an open and bounded set
$\Omega\subseteq\R^3$, the \emph{image domain}, into the set of Borel
\emph{probability measures} $\IP(\IS^2)$ on the two-dimensional unit sphere
$\IS^2$ (or, more generally, on some metric space):
each value $u_x := u(x) \in \IP(\IS^2)$ is a Borel probability measure on
$\IS^2$, and can be viewed as a distribution of directions in
$\R^3$.

Such measures $\mu \in \IP(\IS^2)$, in particular when represented using
density functions, are known as \emph{orientation distribution functions}
(ODFs).
We will keep to the term due to its popularity, although we will be mostly
concerned with \emph{measures} instead of functions on $\IS^2$.
Accordingly, an \emph{ODF-valued image} is a function $u\colon\Omega\to\IP(\IS^2)$.
ODF-valued images appear in reconstruction schemes for diffusion-weighted
magnetic resonance imaging (MRI), such as Q-ball imaging (QBI) {\cite{tuch2004}}
and constrained spherical deconvolution (CSD) {\cite{tournier2004}.

\begin{figure*}
    \begin{tikzpicture}
        \node[inner sep=0] (fibers) at (0,0) {
            \includegraphics[%
                trim=19 8 718 8,clip,width=0.194\textwidth
            ]{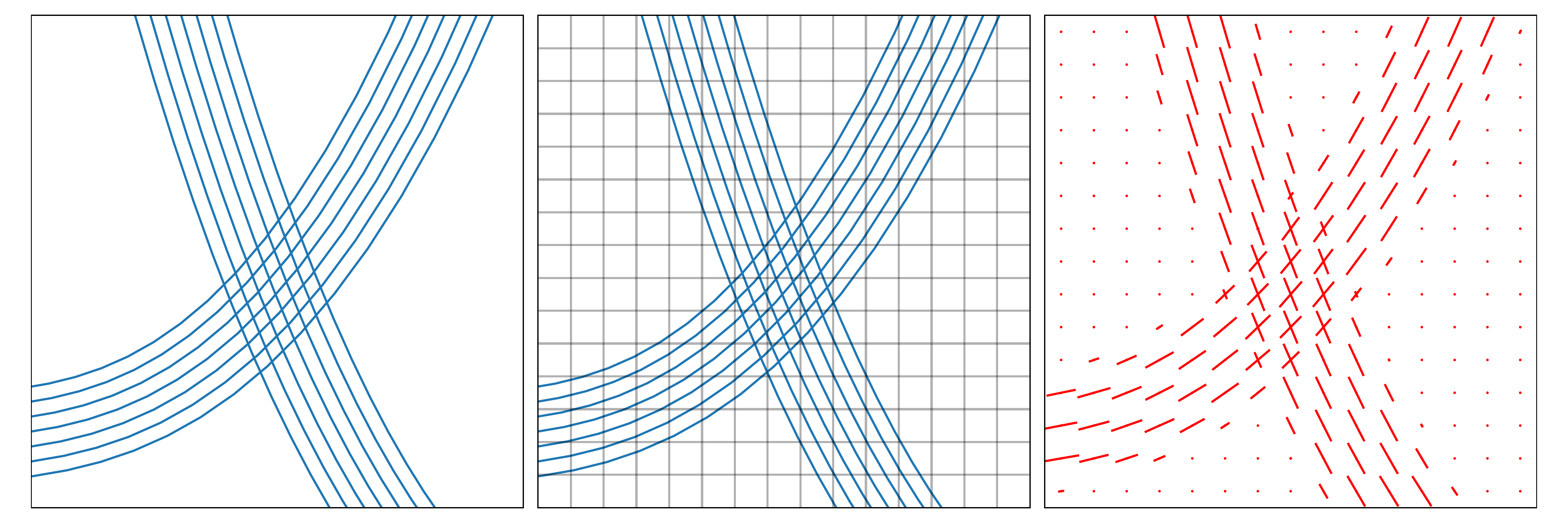}
        };
        \node[inner sep=0,below=1mm of fibers] {
            \includegraphics[%
                trim=718 8 19 8,clip,width=0.194\textwidth
            ]{fig/plot-phantom}
        };
    \end{tikzpicture}
    \includegraphics[%
        trim=220 220 220 220,clip,width=0.396\textwidth
    ]{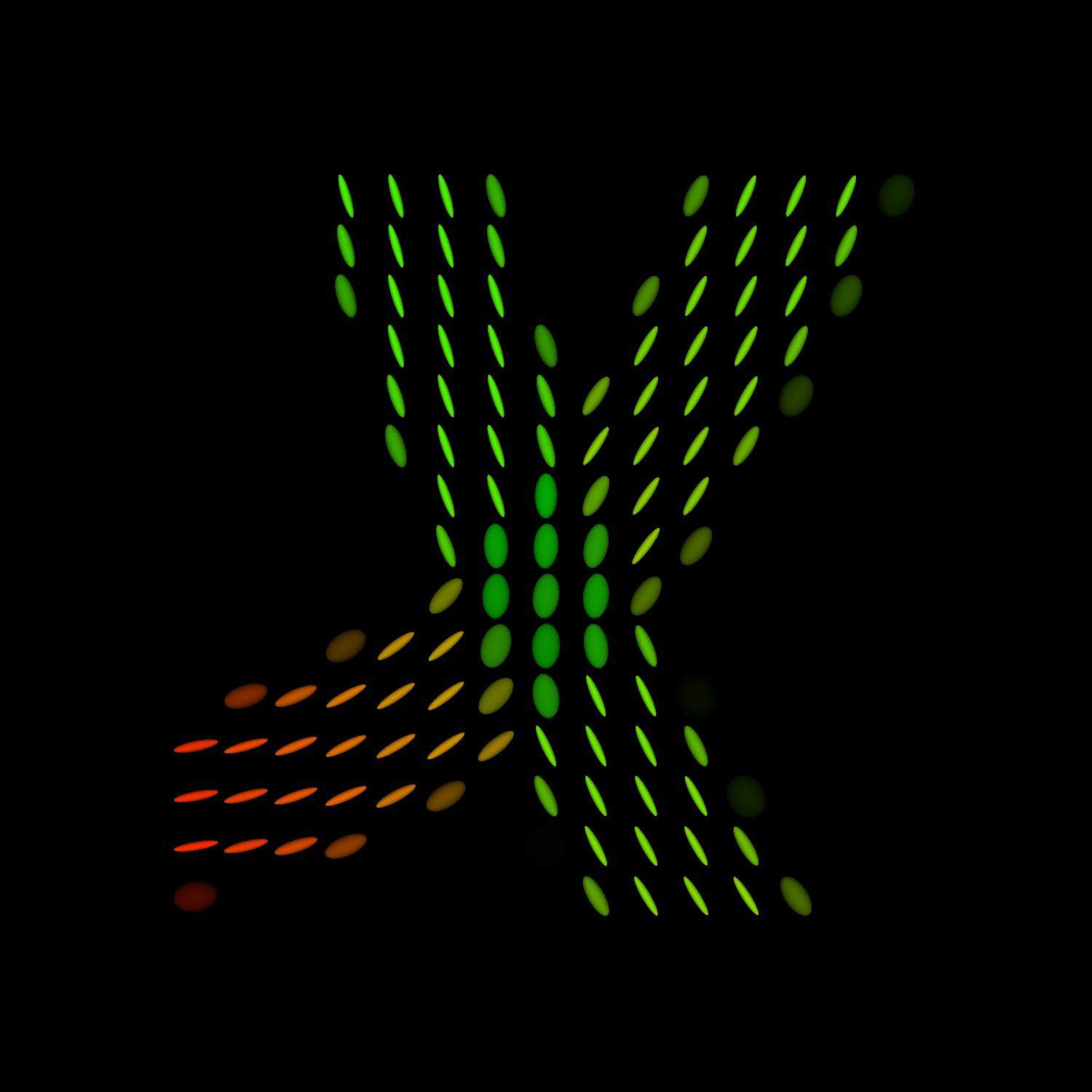}
    \includegraphics[%
        trim=0 0 90 90,clip,width=0.396\textwidth
    ]{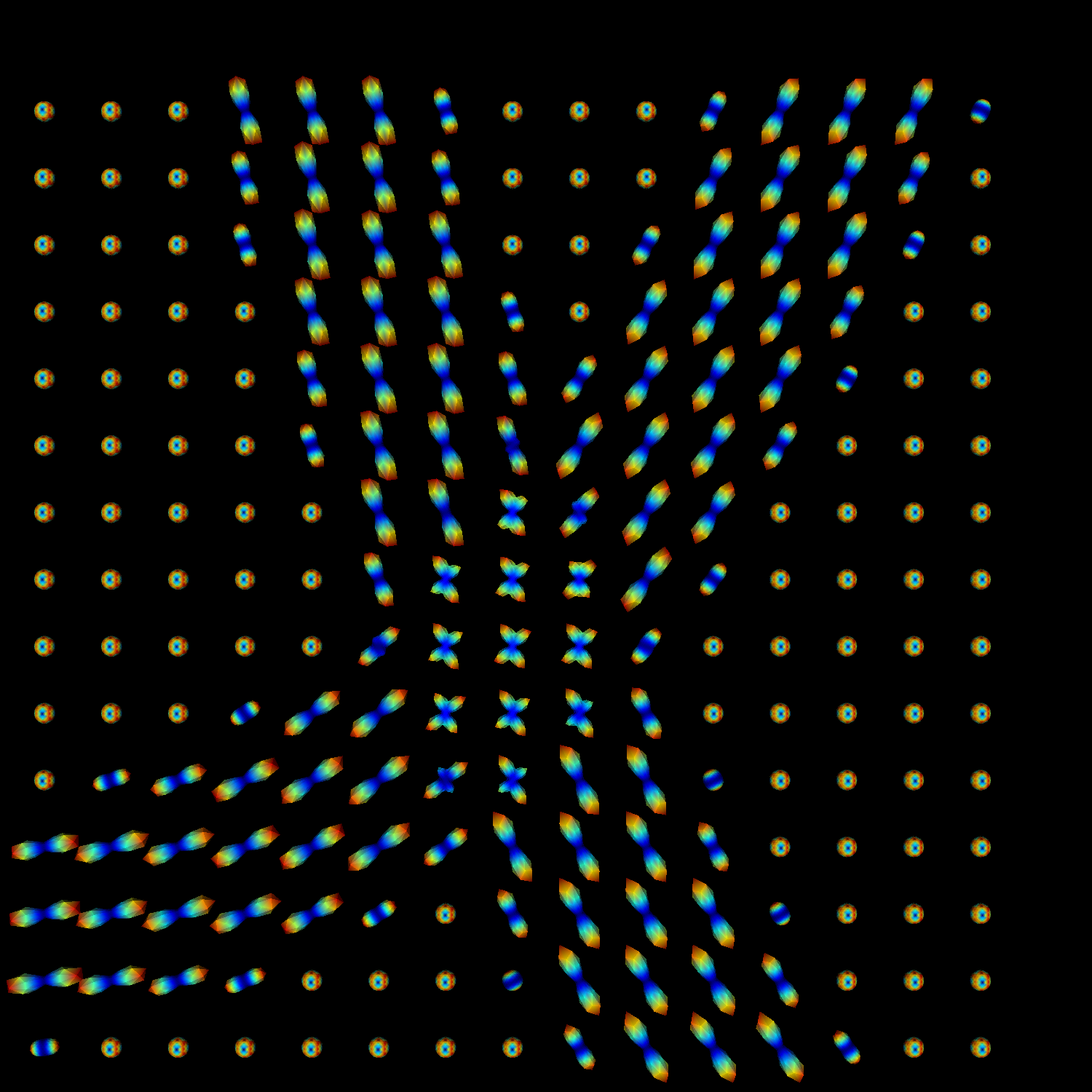}
    \caption{%
        \emph{Top left:} 2-D fiber phantom as described in
            Sect.~\ref{sec:fiber-phantom}.
        \emph{Bottom left:} Peak directions on a $15 \times 15$ grid, derived
            from the phantom and used for the generation of synthetic HARDI data.
         \emph{Center:} The diffusion tensor (DTI) reconstruction approximates
            diffusion directions in a parametric way using tensors, visualized
            as ellipsoids.
        \emph{Right:} The QBI-CSA ODF reconstruction represents fiber
            orientation using probability measures at each point, which allows
            to accurately recover fiber crossings in the center region.
    }\label{fig:phantom-plot}
\end{figure*}
\paragraph{Applications in Diffusion MRI.}
In diffusion-weighted (DW) magnetic resonance imaging (MRI), the diffusivity of
water in biological tissues is measured non-invasively.
In medical applications where tissues exhibit fibrous microstructures, such as
muscle fibers or axons in cerebral white matter, the diffusivity contains
valuable information about the fiber architecture.
For DW measurements, six or more full 3D MRI volumes are acquired with varying
magnetic field gradients that are able to sense diffusion.

Under the assumption of anisotropic Gaussian diffusion, positive definite
matrices (tensors) can be used to describe the diffusion in each voxel.
This model, known as \emph{diffusion tensor imaging} (DTI)~\cite{bass1994},
requires few measurements while giving a good estimate of the main diffusion
direction in the case of well-aligned fiber directions.
However, crossing and branching of fibers at a scale smaller than the voxel
size, also called intra-voxel orientational heterogeneity (IVOH), often occurs
in human cerebral white matter due to the relatively large (millimeter-scale)
voxel size of DW-MRI data.
Therefore, DTI data is insufficient for accurate fiber tract mapping in regions
with complex fiber crossings (Fig.~\ref{fig:phantom-plot}).

More refined approaches are based on \emph{high angular resolution diffusion
imaging} (HARDI)~\cite{tuch2002} measurements that allow for more accurate
restoration of IVOH by increasing the number of applied magnetic field gradients.
Reconstruction schemes for HARDI data yield orientation distribution
functions (ODFs) instead of tensors.
In \emph{Q-ball imaging} (QBI)~\cite{tuch2004}, an ODF is interpreted to be the
marginal probability of diffusion in a given direction \cite{aganj2009}.
In contrast, ODFs in \emph{constrained spherical deconvolution}
(CSD) approaches~\cite{tournier2004}, also denoted \emph{fiber ODFs}, estimate
the density of fibers per direction for each voxel of the volume.

In all of these approaches, ODFs are modelled as antipodally symmetric functions
on the sphere which could be modelled just as well on the projective space (which
is defined to be a sphere where antipodal points are identified).
However, most approaches parametrize ODFs using symmetric spherical harmonics
basis functions which avoids any numerical overhead. 
Moreover, novel approaches~\cite{delp2007,ehric2011,reis2012,CetinKarayumak2018}
allow for asymmetric ODFs to account for intravoxel geometry.
Therefore, we stick to modelling ODFs on a sphere even though our model could
be easily adapted to models on the projective space.
\paragraph{Variational models for orientation distributions.}
As a common denominator, in the above applications, reconstructing orientation
distributions rather than a single orientation at each point allows to recover
directional information of structures -- such as vessels or nerve fibers -- that
may overlap or have crossings:
For a given set of directions $A \subset \IS^2$, the integral $\int_A d u_x(z)$
describes the fraction of fibers crossing the point $x\in\Omega$ that are
oriented in any of the given directions $v\in A$.

However, modeling ODFs as probability measures in a non-parametric way is
surprisingly difficult.
In an earlier conference publication \cite{Vogt2017}, we proposed a new
formulation of the classical total variation seminorm
(TV)~\cite{ambrosio2000,chambolle2010introduction} for nonparametric Q-ball
imaging that allows to formulate the variational restoration model
\begin{equation}\label{eq:variational-problem}
    \inf_{u:\Omega \to \IP(\IS^2)}
        \int_\Omega \rho(x,u_x) \dd x + \lambda\TV_{W_1}(u),
\end{equation}
with various pointwise data fidelity terms
\begin{equation}
    \rho\colon \Omega \times \IP(\IS^2) \to [0,\infty).
\end{equation}
This involved in particular a non-parametric concept of total variation for
ODF-valued functions that is mathematically robust and computationally
feasible:
The idea is to build upon the $\TV$-formulations developed in the context of
functional lifting \cite{lell2013}
\begin{equation}
    \begin{aligned}
    &\TV_{W_1}(u) := \sup\left\{
        \int_\Omega
            \langle -\div p(x,\cdot), u_x \rangle
        \dd x :~\right.\\
        &\left.\phantom{\TV_{W_1}(u) := \sup\{\int_\Omega \langle}
        ~p \in C^1_c(\Omega \times \IS^2; \R^3),
        ~p(x,\cdot) \in \Lip_1(\IS^2; \R^3)
    \right\},
    \end{aligned}\label{eq:liftedTV}
\end{equation}
where $\langle g, \mu \rangle := \int_{\IS^2} g(z)\dd \mu(z)$ whenever $\mu$ is
a measure on $\IS^2$ and $g$ is a real- or vector-valued function on $\IS^2$.

One distinguishing feature of this approach is that it is applicable
to arbitrary Borel probability measures.
In contrast, existing mathematical frameworks for QBI and CSD generally follow
the standard literature on the physics of MRI \cite[p. 330]{Callaghan91} in
assuming ODFs to be given by a \emph{probability density function} in
$L^1(\IS^2)$, often with an explicit parametrization.

As an example of one such approach, we point to the fiber continuity regularizer
proposed in \cite{reisert2013} which is defined for ODF-valued functions $u$
where, for each $x \in \Omega$, the measure $u_x$ can be represented by a 
probability density function $z \mapsto u_x(z)$ on $\IS^2$:
\begin{equation}
    R_{\mathrm{FC}}(u) := \int_\Omega \int_{\IS^2}
        (z \cdot \nabla_x u_x(z))^2
    \dd z \dd x
\end{equation}
Clearly, a rigorous generalization of this functional to measure-valued
functions for arbitrary Borel probability measures is not straightforward.

\begin{figure*}%
    \includegraphics[%
        trim=5 246 12 11,clip,width=\textwidth
    ]{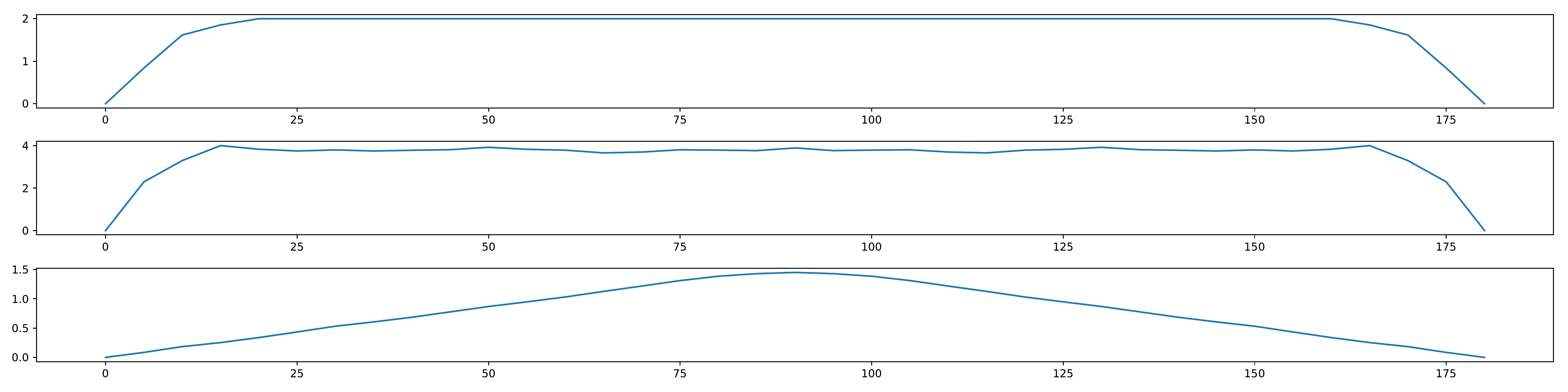}

    \vspace{2mm}
    \includegraphics[%
        trim=5 12 12 242,clip,width=\textwidth
    ]{fig/lp-w1-result}

    \vspace{2mm}
    \includegraphics[%
        trim=502 475 1465 475,clip,width=0.032\textwidth
    ]{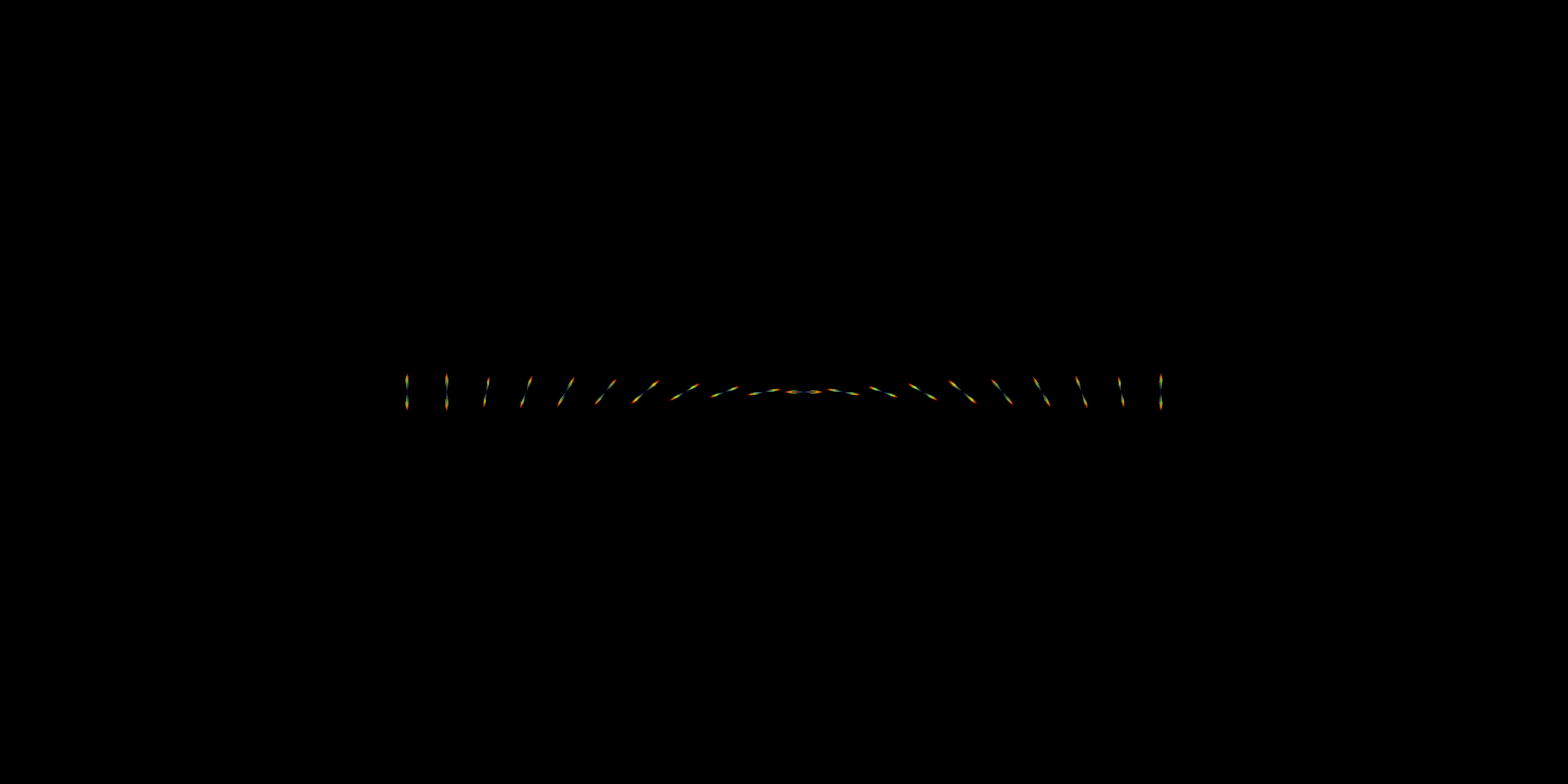}\hfill
    \includegraphics[%
        trim=555 475 505 475,clip,width=0.917\textwidth
    ]{fig/lp-w1-data}\hspace{0.031\textwidth}
    \caption{%
        \emph{Horizontal axis:} Angle of main diffusion direction relative
        to the reference diffusion profile in the bottom left corner.
        \emph{Vertical axis:} Distances of the ODFs in the bottom row to
        the reference ODF in the bottom left corner ($L^1$-distances in the
        top row and $W^1$-distance in the second row).
        $L^1$-distances do not reflect the linear change in direction, whereas
        the $W^1$-distance exhibits an almost-linear profile.
        $L^p$-distances for other values of $p$ (such as $p=2$) show a behavior
        similar to $L^1$-distances.
    }\label{fig:lpw1-plot}
\end{figure*}

While practical, the probability density-based approach raises some modeling
questions, which lead to deeper mathematical issues.
In particular, comparing probability densities using the popular
$L^p$-norm-based data fidelity terms -- in particular the squared $L^2$-norm
-- does
not incorporate the structure naturally carried by probability densities such
as nonnegativity and unit total mass, and ignores metric information about
$\IS^2$.

To illustrate the last point, assume that two probability measures are given in
terms of density functions $f,g \in L^p(\IS^2)$ satisfying
$\supp(f) \cap \supp(g) = \emptyset$, i.e., having disjoint support on $\IS^2$.
Then $\|f - g\|_{L^p} = \|f\|_{L^p} + \|g\|_{L^p}$, irrespective of the size and
relative position of the supporting sets of $f$ and $g$ on $\IS^2$.

One would prefer to use statistical metrics such as optimal transport metrics
\cite{villani2009} that properly take into account distances on the underlying
set $\IS^2$ (Fig.~\ref{fig:lpw1-plot}).
However, replacing the $L^p$-norm with such a metric in density-based
variational imaging formulations will generally lead to ill-posed minimization
problems, as the minimum might not be attained in $L^p(\IS^2)$, but possibly in
$\IP(\IS^2)$ instead.

Therefore, it is interesting to investigate whether one can derive a mathematical basis for variational
image processing with ODF-valued functions without making assumptions about the
para\-metrization of ODFs nor assuming ODFs to be given by density functions.
\subsection{Contribution}
Building on the preliminary results published in the conference publication
\cite{Vogt2017}, we derive a rigorous mathematical framework
(Sect.~\ref{sec:theory} and Appendices) for a generalization of the total
variation seminorm formulated in \eqref{eq:liftedTV} to Banach space-valued%
\footnote{%
    Here and throughout the paper, we use ``Banach space-valued'' as a synonym
    for ``taking values in a Banach space'' even though we acknowledge the
    ambiguity carried by this expression.
    Similarly, ``metric space-valued'' is used in \cite{Ambrosio1990} and
    ``manifold-valued'' in \cite{BBSW16}.
}
and, as a special case, ODF-valued functions (Sect.~\ref{sec:rigorous-tv}).

Building on this framework, we show existence of minimizers to
\eqref{eq:variational-problem} (Thm.~\ref{thm:existence}) and discuss properties
of $\TV$ such as rotational invariance (Prop.~\ref{prop:rot-invariance}) and the
behavior on cartoon-like jump functions (Prop.~\ref{prop:cartoon-tv}).

We demonstrate that our framework can be numerically implemented
(Sect.~\ref{sec:numerics}) as a primal-dual saddle-point problem involving only
convex functions.
Applications to synthetic and real-world data sets show significant reduction of
noise as well as qualitatively convincing results when combined with existing
ODF-based imaging approaches, including Q-ball and CSD
(Sect.~\ref{sec:results}).

Details about the functional-analytic and measure-theoretic background of our
theory are given in Appendix~\ref{apdx:background}.
There, well-definedness of the $\TV$-seminorm and of variational problems of the
form \eqref{eq:variational-problem} is established by carefully considering
measurability of the functions involved (Lemmas \ref{lem:measurability1} and
\ref{lem:rho-measurable}).
Furthermore, a func\-tio\-nal-analytic explanation for the dual structure that
is inherent in \eqref{eq:liftedTV} is given.
\subsection{Related Models}
The high angular resolution of HARDI results in a large amount of noise compared
with DTI. Moreover, most QBI and CSD models reconstruct the ODFs in each
voxel separately.
Consequently, HARDI data is a particularly interesting target for post-processing
in terms of denoising and regularization in the sense of contextual processing.
Some techniques apply a total variation or diffusive regularization to the HARDI
signal before ODF reconstruction \cite{mcgraw2009,kim2010,duits2011,becker2012}
and others regularize in a post-processing step \cite{delp2007,duits2013,weinm2016}.
\subsubsection{Variational Regularization of DW-MRI Data}
A Mumford-Shah model for edge-preserving restoration of Q-ball data
was introduced in \cite{weinm2016}.
There, jumps were penalized using the Fisher-Rao metric which depends on a
parametrization of ODFs as discrete probability distribution functions on
sampling points of the sphere.
Furthermore, the Fisher-Rao metric does not take the metric structure of $\IS^2$
into consideration and is not amenable to biological
interpretations~\cite{ncube2011}.
Our formulation avoids any parametrization-induced bias.

Recent approaches directly incorporate a regularizer into the reconstruction
scheme:
Spatial TV-based regularization for Q-ball imaging has been proposed in
\cite{ouyang2014}.
However, the TV formulation proposed therein again makes use of the underlying
parametrization of ODFs by spherical harmonics basis functions.
Similarly, DTI-based models such as the second-order model for regularizing
general manifold-valued data~\cite{BBSW16} make use of an explicit approximation
using positive semidefinite matrices, which the proposed model avoids.

The application of spatial regularization to CSD reconstruction is known to
significantly enhance the results \cite{daducci2014}.
However, total variation \cite{canales2015} and other regularizers
\cite{hohage2015} are based on a representation of ODFs by square-integrable
probability density functions instead of the mathematically more general
probability measures that we base our method on.
\subsubsection{Regularization of DW-MRI by Linear Diffusion}
In another approach, the orientational part of ODF-valued images is
included in the image domain, so that images are identified with functions
$U\colon \R^3 \times \IS^2 \to \R$ that allow for contextual processing via
PDE-based models on the space of positions and orientation or, more precisely,
on the group $SE(3)$ of 3D rigid motions.
This technique comes from the theory of stochastic processes on the coupled space
$\R^3 \times \IS^2$.
In this context, it has been applied to the problems of contour completion
\cite{siddiqi2013} and contour enhancement \cite{duits2011,duits2013}.
Its practical relevance in clinical applications has been demonstrated 
\cite{prckovska2015}.

This approach has been used to enhance the quality of CSD as a prior in a
variational formulation \cite{reisert2013} or in a post-processing step
\cite{portegies2015} that also includes additional angular regularization.
Due to the linearity of the underlying linear PDE, convolution-based explicit solution
formulas are available \cite{duits2011,portegies2017}. Implemented efficiently \cite{meesters2016a,meesters2016b}, they outperform our more computationally demanding model, which is not tied to the specific application of DW-MRI, but allows arbitrary metric spaces.
Furthermore, nonlinear Perona and Malik extensions to this technique have been
studied \cite{creusen2013} that do not allow for explicit solutions.

As an important distinction, in these approaches, spatial location and orientation are coupled in the
regularization.
Since our model starts from the more general setting of measure-valued functions
on an arbitrary metric space (instead of only $\IS^2$), it does not currently realize an
equivalent coupling. An extension to anisotropic total variation for measure-valued functions might close this gap in the future.

In contrast to these diffusion-based methods, our approach is able to preserve
edges by design, even though the coupling of positions and orientations is
able to make up for this shortcoming at least in part since edges in DW-MRI are,
most of the time, oriented in parallel to the direction of diffusion.
Furthermore, the diffusion-based methods are formulated for square-integrable
density functions, excluding point masses. Our method avoids this limitation by operating on
mathematically more general probability measures.
\subsubsection{Other Related Theoretical Work}
Variants of the Kantorovich-Rubinstein formulation of the Wasserstein
distance that appears in our framework have been applied in \cite{lell2014} and,
more recently, in \cite{fitschen2016,fitschen2017} to the problems of real-,
RGB- and manifold-valued image denoising.

Total variation regularization for functions on the space of positions and
orientations was recently introduced in \cite{chambolle2017} based on
\cite{mirebeau2017}.
Similarly, the work and toolbox in \cite{skibbe2017} is concerned with
the implementation of so-called \emph{orientation fields} in 3D image processing.

A Dirichlet energy for measure-valued functions based on Wasserstein metrics 
was recently developed in the context of harmonic mappings in \cite{lavenant2017}
which can be interpreted as a diffusive ($L^2$) version of our proposed 
($L^1$) regularizer.

Our work is based on the conference publication~\cite{Vogt2017}, where a
non-parametric Wasserstein-total variation regularizer for Q-ball data is
proposed.
We embed this formulation of TV into a significantly more general definition of
TV for Banach space-valued functions.

In the literature, Banach space-valued functions of bounded variation mostly
appear as a special case of metric space-valued functions of bounded variation
(BV) as introduced in \cite{Ambrosio1990}.
Apart from that, the case of one-dimensional domains attracts some attention
\cite{duchon2011} and the case of Banach space-valued BV-functions defined on
a metric space is studied in \cite{miranda2002}.

In contrast to these approaches, we give a definition of Banach space-valued BV
functions that live on a finite-dimensional domain.
In analogy with the real-valued case, we formulate the TV seminorm by
duality, inspired by the functional-analytic framework from the theory of
functional lifting~\cite{IonescuTulcea1969} as used in the theory
of Young-measures \cite{Ball1989}.

Due to the functional-analytic approach, our model does not depend on the
specific parametrization of the ODFs and can be combined with the QBI and CSD
frameworks for ODF reconstruction from HARDI data, either in a post-processing
step or during reconstruction.
Combined with suitable data fidelity terms such as least-squares or Wasserstein
distances, it allows for an efficient implementation using state-of-the-art
primal-dual methods.
\section{A Mathematical Framework for
         Measure-Valued Functions}\label{sec:theory}
Our work is motivated by the study of ODF-valued functions
$u\colon \Omega \to \IP(\IS^2)$ for $\Omega \subset \R^3$ open and bounded.
However, from an abstract viewpoint, the unit sphere $\IS^2 \subset \R^3$
equipped with the metric induced by the Riemannian manifold structure
\cite{lee1997} -- i.e., the distance between two points is the arc length of
the great circle segment through the two points -- is simply a particular
example of a compact metric space.

As it turns out, most of the analysis only relies on this property.
Therefore, in the following we generalize the setting of ODF-valued functions
to the study of functions taking values in the space of Borel probability
measures on an \emph{arbitrary} compact metric space (instead of $\IS^2$).

More precisely, throughout this section, let
\begin{enumerate}
\item $\Omega \subset \R^d$ be an open and bounded set, and let
\item $(X,d)$ be a compact metric space, e.g.,
    a compact Riemannian manifold equipped with the commonly-used metric
    induced by the geodesic distance (such as $X = \IS^2$).
\end{enumerate}
Boundedness of $\Omega$ and compactness of $X$ are not required by all of the
statements below.
However, as we are ultimately interested in the case of $X = \IS^2$ and
rectangular image domains, we impose these restrictions.
Apart from DW-MRI, one natural application of this generalized setting are
two-dimensional ODFs where $d = 2$ and $X = \IS^1$ which is similar
to the setting introduced in \cite{chambolle2017} for the edge enhancement of
color or grayscale images.

The goal of this section is a mathematically well-defined formulation of $\TV$
as given in \eqref{eq:liftedTV} that exhibits all the properties that the
classical total variation seminorm is known for:
anisotropy (Prop. \ref{prop:rot-invariance}), preservation of edges and
compatibility with piecewise-constant signals (Prop. \ref{prop:cartoon-tv}).
Furthermore, for variational problems as in \eqref{eq:variational-problem}, we
give criteria for the existence of minimizers (Theorem \ref{thm:existence}) and
discuss (non-)uniqueness (Prop. \ref{prop:non-uniqueness}).

A well-defined formulation of $\TV$ as given in \eqref{eq:liftedTV} requires a
careful inspection of topological and functional analytic concepts from optimal
transport and general measure theory.
For details, we refer the reader to the elaborate Appendix~\ref{apdx:background}.
Here, we only introduce the definitions and notation needed for the statement
of the central results.
\subsection{Definition of $\TV$}\label{sec:rigorous-tv}
We first give a definition of $\TV$ for Banach space-valued functions (i.e.,
functions that take values in a Banach space), which a definition of $\TV$
for measure-valued functions will turn out to be a special case of.

For weakly measurable (see Appendix~\ref{apdx:bv-bv}) functions
$u\colon \Omega \to V$ with values in a Banach space $V$ (later, we will replace
$V$ by a space of measures), we define, extending the formulation of
$\TV_{W_1}$ introduced in \cite{Vogt2017},
\begin{equation}\label{eq:bv-valued-tv}
    \begin{aligned}
    &\TV_{V}(u) := \sup\left\{
        \int_\Omega
            \langle -\div p(x), u(x) \rangle
        \dd x :~\right.\\
        &\left.\phantom{\TV_{V}(u) := \sup\{\int_\Omega \langle}
        ~p \in C_c^1(\Omega, (V^*)^d),
        ~\forall x \in \Omega\colon \|p(x)\|_{(V^*)^d} \leq 1
    \right\}.
    \end{aligned}
\end{equation}
By $V^*$, we denote the (topological) dual space of $V$, i.e., $V^*$ is the set
of bounded linear operators from $V$ to $\R$.
The criterion $p \in C_c^1(\Omega, (V^*)^d)$ means that $p$ is a compactly
supported function on $\Omega \subset \R^d$ with values in the Banach space
$(V^*)^d$ and the directional derivatives $\partial_i p\colon \Omega \to (V^*)^d$,
$1 \leq i \leq d$, (in Euclidean coordinates) lie in $C_c(\Omega, (V^*)^d)$.
We write
\begin{equation}
    \div p(x) := \sum_{i=1}^d \partial_i p_i(x).
\end{equation}
Lemma~\ref{lem:measurability1} ensures that the integrals in
\eqref{eq:bv-valued-tv} are well-defined and Appendix~\ref{apdx:product-norms}
discusses the choice of the product norm $\|\cdot\|_{(V^*)^d}$.
\paragraph{Measure-valued functions.}
Now we want to apply this definition to measure-valued functions 
$u\colon \Omega \to \IP(X)$, where $\IP(X)$ is the set of Borel probability
measures supported on $X$.

The space $\IP(X)$ equipped with the Wasserstein metric $W_1$ from the theory of
optimal transport is isometrically embedded into
the Banach space $V = \KR(X)$ (the \emph{Kantorovich-Rubinstein space}) whose
dual space is the space $V^* = \Lip_0(X)$ of Lipschitz-continuous functions on
$X$ that vanish at an (arbitrary but fixed) point $x_0 \in X$.
This setting is introduced in detail in Appendix~\ref{apdx:kr}.
Then, for $u\colon \Omega \to \IP(X)$, definition \eqref{eq:bv-valued-tv}
comes back to \eqref{eq:liftedTV} or, more precisely,
\begin{equation}\label{eq:measure-valued-tv}
    \begin{aligned}
    &\TV_{\KR}(u) := \sup\left\{
        \int_\Omega
            \langle -\div p(x), u(x) \rangle
        \dd x :~\right.\\
        &\left.\phantom{\TV_{\KR}(u) := \sup\{\int_\Omega\langle}\!\!
        p \in C_c^1(\Omega, [\Lip_0(X)]^d), ~\|p(x)\|_{[\Lip_0(X)]^d} \leq 1
    \right\},
    \end{aligned}
\end{equation}
where the definition of the product norm $\|\cdot\|_{[\Lip_0(X)]^d}$ is
discussed in Appendix~\ref{apdx:lip-product-norm}.
\subsection{Properties of $\TV$}
In this section, we show that the properties that the classical total variation
seminorm is known for continue to hold for definition \eqref{eq:bv-valued-tv}
in the case of Banach space-valued functions.
\paragraph{Cartoon functions.}
A reasonable demand is that the new formulation should behave similarly to the
classical total variation on cartoon-like jump functions $u\colon \Omega \to V$,
\begin{equation}\label{eq:jump-u}
    u(x) := \begin{cases}
        u^+, & x\in U, \\
        u^-, & x \in \Omega \setminus U, \\
    \end{cases}
\end{equation}
for some fixed measurable set $U \subset \Omega$ with smooth boundary
$\partial U$, and $u^+, u^- \in V$.
The classical total variation assigns to such functions  a penalty of
\begin{equation}
    \IH^{d-1}(\partial U)\cdot \|u^+ - u^-\|_V,
\end{equation}
where the Hausdorff measure $\IH^{d-1}(\partial U)$ describes the length or
area of the jump set.
The following proposition, which generalizes \cite[Prop.~1]{Vogt2017},
provides conditions on the norm $\|\cdot\|_{(V^*)^d}$ which guarantee this
behavior.
\begin{proposition}\label{prop:cartoon-tv}
    Assume that $U$ is compactly contained in $\Omega$ with $C^1$-boundary
    $\partial U$.
    Let $u^+, u^- \in V$ and let $u\colon \Omega \to V$ be defined
    as in \eqref{eq:jump-u}.
    If the norm $\|\cdot\|_{(V^*)^d}$ in \eqref{eq:bv-valued-tv} satisfies
    \begin{align}
        \label{eq:product-norm-lower-bound}
        &\left|\textstyle{\sum_{i=1}^d x_i} \langle p_i, v \rangle\right|
            \leq \|x\|_2 \|p\|_{(V^*)^d} \|v\|_V, \\
        \label{eq:product-norm-upper-bound}
        &\|(x_1 q, \dots, x_d q)\|_{(V^*)^d} \leq \|x\|_2 \|q\|_{V^*}
    \end{align}
    whenever $q \in V^*$, $p \in (V^*)^d$, $v \in V$, and $x \in \R^d$, then
    \begin{equation}\label{eq:tv-jump-penalty}
        \TV_{V}(u) = \IH^{d-1}(\partial U) \cdot \|u^+ - u^-\|_V.
    \end{equation}
\end{proposition}
\begin{proof}
    See Appendix~\ref{apdx:cartoon-tv}. \qed
\end{proof}
\paragraph{Rotational invariance.}
Property \eqref{eq:tv-jump-penalty} is inherently rotationally invariant:
we have $\TV_V(u) = \TV_V(\tilde u)$ whenever $\tilde u(x) := u(Rx)$ for some
$R \in SO(d)$ and $u$ as in \eqref{eq:jump-u}, with the domain $\Omega$ rotated
accordingly.
The reason is that the jump size is the same everywhere along the edge $\partial U$.
More generally, we have the following proposition:
\begin{proposition}\label{prop:rot-invariance}
    Assume that $\|\cdot\|_{(V^*)^d}$ satisfies the rotational invariance
    property
    \begin{equation}\label{eq:product-norm-rot-inv}
        \|p\|_{(V^*)^d} = \|R p\|_{(V^*)^d}
        \quad\forall p \in (V^*)^d, R \in SO(d),
    \end{equation}
    where $Rp \in (V^*)^d$ is defined via
    \begin{equation}
        (Rp)_i = \sum_{j=1}^d R_{ij} p_j \in V^*.
    \end{equation}
    Then $\TV_V$ is rotationally invariant, i.e., $\TV_V(u) = \TV_V(\tilde u)$
    whenever $u \in L_w^\infty(\Omega, V)$ and $\tilde u(x) := u(Rx)$ for some
    $R \in SO(d)$.
\end{proposition}
\begin{proof}[Prop.~\ref{prop:rot-invariance}]
    See Appendix~\ref{apdx:rot-invariance}. \qed
\end{proof}
\subsection{$\TV_{\KR}$ as a Regularizer in Variational Problems}
This section shows that, in the case of measure-valued functions
$u\colon \Omega \to \IP(X)$, the functional $\TV_{\KR}$ exhibits a
regularizing property, i.e., it establishes existence of minimizers.

For $\lambda \in [0,\infty)$ and $\rho\colon \Omega \times \IP(X) \to [0,\infty)$
fixed, we consider the functional
\begin{equation}\label{eq:T-definition}
    T_{\rho,\lambda}(u)
    := \int_\Omega \rho(x, u(x)) \dd x + \lambda \TV_{\KR}(u).
\end{equation}
for $u\colon \Omega \to \IP(X)$.
Lemma~\ref{lem:rho-measurable} in Appendix~\ref{apdx:existence} makes sure that
the integrals in \eqref{eq:T-definition} are well-defined.

Then, minimizers of the energy \eqref{eq:T-definition} exist in the
following sense:
\begin{theorem}\label{thm:existence}
    Let $\Omega \subset \R^d$ be open and bounded, let $(X,d)$ be a
    compact metric space and assume that $\rho$ satisfies the assumptions from
    Lemma \ref{lem:rho-measurable}.
    Then the variational problem
    \begin{equation}\label{eq:trhoinprop}
        \inf_{u \in L_w^\infty(\Omega, \IP(X))} T_{\rho,\lambda}(u)
    \end{equation}
    with the energy
    \begin{equation}
        T_{\rho,\lambda}(u)
        := \int_\Omega \rho(x, u(x)) \dd x + \lambda \TV_{\KR}(u).
    \end{equation}
    as in \eqref{eq:T-definition} admits a (not necessarily unique) solution.
\end{theorem}

\begin{proof}
    See Appendix~\ref{apdx:existence}. \qed
\end{proof}

Non-uniqueness of minimizers of \eqref{eq:T-definition} is clear for
pathological choices such as $\rho \equiv 0$.
However, there are non-trivial cases where uniqueness fails to hold:

\begin{proposition}\label{prop:non-uniqueness}
    Let $X = \{0,1\}$ be the metric space consisting of two discrete points of
    distance $1$ and define $\rho(x,\mu) := W_1(f(x),\mu)$ where
    \begin{equation}
        f(x) := \begin{cases}
            \delta_1, & x \in \Omega \setminus U, \\
            \delta_0, & x \in U,
        \end{cases}
    \end{equation}
    for a non-empty subset $U \subset \Omega$ with $C^1$ boundary.
    Assume the coupled norm \eqref{eq:lip-norm-Rd} on $[\Lip_0(X)]^d$ in the
    definition \eqref{eq:measure-valued-tv} of $\TV_{\KR}$.
   
    Then there is a one-to-one correspondence between feasible solutions $u$ of
    problem \eqref{eq:trhoinprop} and feasible solutions $\tilde{u}$ of the
    classical $L^1$-$\TV$ functional
    \begin{align}
        \inf_{\tilde u \in L^1(\Omega,[0,1])} \tilde{T}_{\lambda}(u),\;
        \tilde{T}_{\lambda}(u):=
            \|\mathbf{1}_U - \tilde u\|_{L^1} + \lambda \TV(\tilde u)\label{eq:l1special}
    \end{align}
    via the mapping
    \begin{equation}
    u(x) = \tilde u(x) \delta_0 + (1 - \tilde u(x)) \delta_1.
    \end{equation}
    Under this mapping $\tilde{T}_{\lambda}(\tilde{u}) = T_{\rho,\lambda}(u)$
    holds, so that the problems \eqref{eq:trhoinprop} and \eqref{eq:l1special}
    are equivalent.

    Furthermore, there exists $\lambda > 0$ for which the minimizer of
    $T_{\rho,\lambda}$ is not unique.
\end{proposition}
\begin{proof}
    See Appendix~\ref{apdx:non-uniqueness}. \qed
\end{proof}
\subsection{Application to ODF-Valued Images}
For ODF-valued images, we consider the special case $X = \IS^2$ equipped with
the metric induced by the standard Riemannian manifold structure on $\IS^2$, and
$\Omega \subset \R^3$.

Let $f \in L_w^\infty(\Omega, \IP(\IS^2))$ be an ODF-valued image and denote by
$W_1$ the Wasserstein metric from the theory of optimal transport (see equation
\eqref{eq:krdual} in Appendix~\ref{apdx:kr}).
Then the function
\begin{equation}
    \rho(x,\mu) := W_1(f(x),\mu), ~x \in \Omega, ~\mu \in \IP(\IS^2),
\end{equation}
satisfies the assumptions in Lemma~\ref{lem:rho-measurable} and hence
Theorem~\ref{thm:existence} (see Appendix~\ref{apdx:existence}).

For denoising of an ODF-valued function $f$ in a postprocessing step after ODF
reconstruction, similar to~\cite{Vogt2017} we propose to solve the variational
minimization problem
\begin{equation}\label{eq:W1TV-functional}
    \inf_{u:\Omega \to \IP(\IS^2)}
        \int_\Omega W_1(f(x),u(x)) \dd x + \lambda \TV_{\KR}(u)
\end{equation}
using the definition of $\TV_{\KR}(u)$ in \eqref{eq:measure-valued-tv}.

The following statement shows that this in fact penalizes jumps in $u$ by the
Wasserstein distance as desired, correctly taking the metric structure of
$\IS^2$ into account.
\begin{corollary}
    Assume that $U$ is compactly contained in $\Omega$ with $C^1$-boundary
    $\partial U$.
    Let the function $u\colon \Omega \to \IP(\IS^2)$ be defined
    as in \eqref{eq:jump-u} for some $u^+, u^- \in \IP(\IS^2)$.
    Choosing the norm \eqref{eq:lip-norm-Rd} (or \eqref{eq:product-norm}
    with $s=2$) on the product space $\Lip(\IS^2)^d$, we have
    \begin{equation}
        \TV_{\KR}(u) = \IH^{d-1}(\partial U) \cdot W_1(u^+, u^-).
    \end{equation}
\end{corollary}
The corollary was proven directly in \cite[Prop.~1]{Vogt2017}.
In the functional-analytic framework established above, it now follows as a simple
corollary to Proposition~{\ref{prop:cartoon-tv}}.

Moreover, beyond the theoretical results given in \cite{Vogt2017}, we now have
a rigorous framework that ensures measurability of the integrands in
\eqref{eq:W1TV-functional}, which is crucial for well-definedness.
Furthermore, Theorem~\ref{thm:existence} on the existence of minimizers provides
an important step in proving well-posedness of the variational model
\eqref{eq:W1TV-functional}.
\section{Numerical Scheme}\label{sec:numerics}
As in \cite{Vogt2017}, we closely follow the discretization scheme from
\cite{lell2013} in order to formulate the problem in a saddle-point form that is
amenable to standard primal-dual algorithms
\cite{chamb2011,pock2011,goldstein2013,goldstein2015a,goldstein2015b}.
\subsection{Discretization}
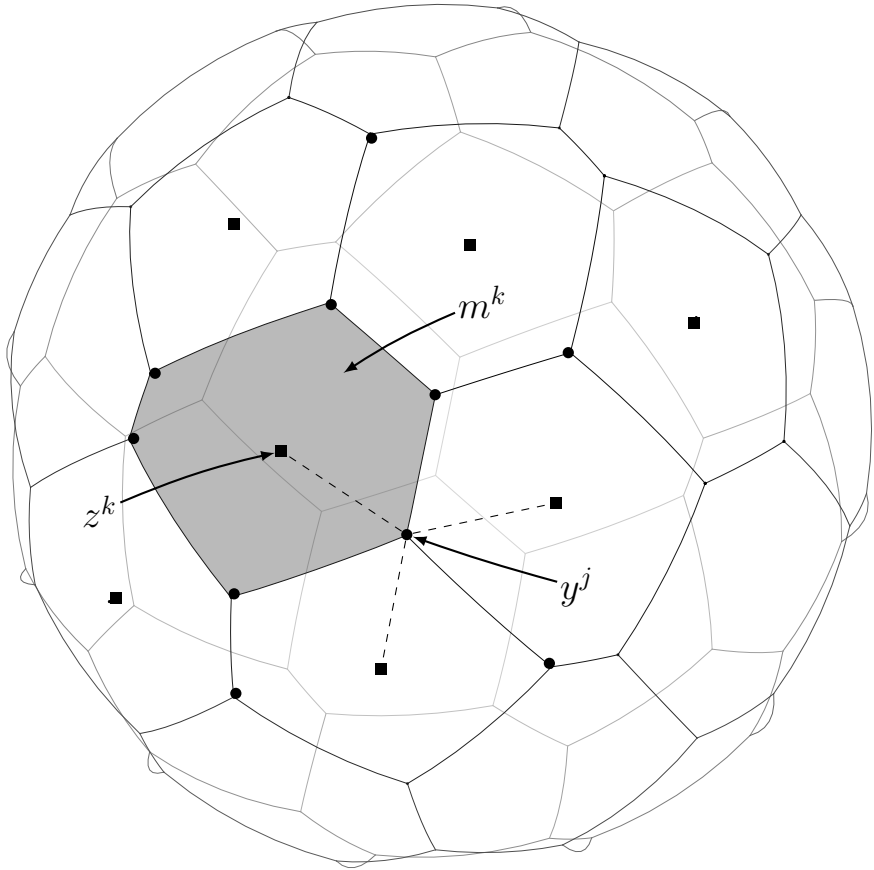
\begin{figure}\centering
    \begin{tikzpicture}
        \node[inner sep=7pt] (sphere-plot) at (0,0) {
            \includegraphics[%
                trim=50 50 50 50,width=0.9\textwidth
            ]{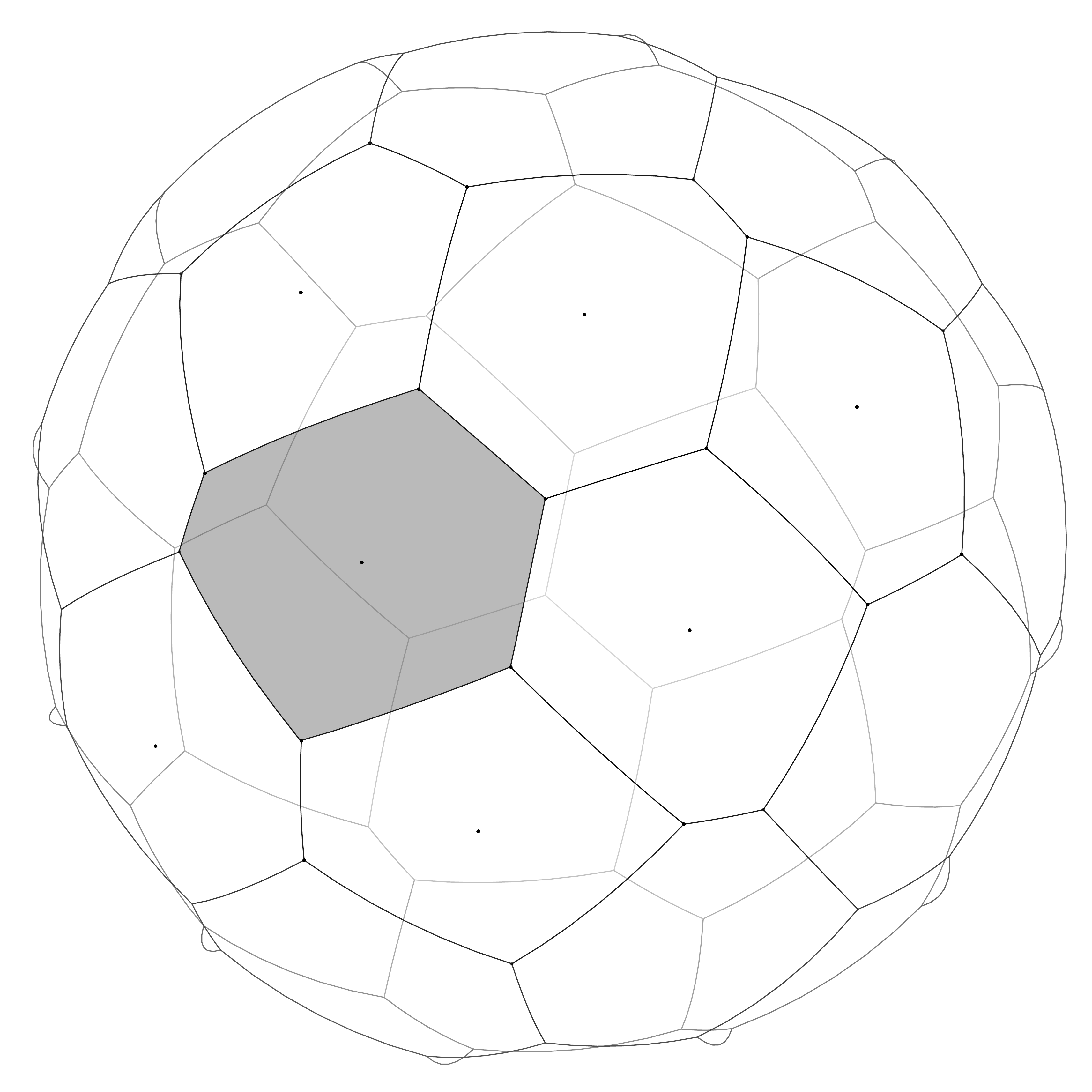}
        };
        
        \ExtractImgDims{sphere-plot};
        \begin{scope}[x=\imgw,y=\imgh,xshift=\imgx,yshift=\imgy]
            \node[circle,fill,inner sep=1.5pt] at (0.267,0.198) {};
            \node[circle,fill,inner sep=1.5pt] at (0.632,0.233) {};
            \node[circle,fill,inner sep=1.5pt] at (0.654,0.595) {};
            \node[circle,fill,inner sep=1.5pt] at (0.425,0.845) {};
            \node[circle,fill,inner sep=1.5pt] at (0.265,0.314) {};
            \node[circle,fill,inner sep=1.5pt] at (0.148,0.495) {};
            \node[circle,fill,inner sep=1.5pt] at (0.173,0.571) {};
            \node[circle,fill,inner sep=1.5pt] at (0.378,0.651) {};
            \node[circle,fill,inner sep=1.5pt] at (0.499,0.546) {};
            \node[circle,fill,inner sep=1.5pt] (yj) at (0.466,0.383) {};
            \node[below right=0.03 and 0.17 of yj,inner sep=1pt] (yj-label) {\Large$y^j$};
            \draw[thick,>=latex,->] (yj-label) to[bend left=2] (yj);

            \node[inner sep=0] at (0.128,0.309) {\rule{4.5pt}{4.5pt}};
            \node[inner sep=0] at (0.54,0.72) {\rule{4.5pt}{4.5pt}};
            \node[inner sep=0] at (0.265,0.745) {\rule{4.5pt}{4.5pt}};
            \node[inner sep=0] at (0.80,0.63) {\rule{4.5pt}{4.5pt}};
            \node[inner sep=0] (zk1) at (0.436,0.226) {\rule{4.5pt}{4.5pt}};
            \node[inner sep=0] (zk2) at (0.64,0.42) {\rule{4.5pt}{4.5pt}};
            \node[inner sep=0] (zk) at (0.32,0.48) {\rule{4.5pt}{4.5pt}};
            \node[below left=0.04 and 0.18 of zk,inner sep=1pt] (zk-label) {\Large$z^k$};
            \draw[thick,>=latex,->] (zk-label) to[bend left=5] (zk);

            \draw[dashed] (yj) to (zk);
            \draw[dashed] (yj) to (zk1);
            \draw[dashed] (yj) to (zk2);

            \node[inner sep=0] (mk) at (0.39,0.57) {};
            \node[above right=0.06 and 0.13 of mk,inner sep=1pt] (mk-label) {\Large$m^k$};
            \draw[thick,>=latex,->] (mk-label) to[bend right=5] (mk);
        \end{scope}
    \end{tikzpicture}
    \caption{%
        Discretization of the unit sphere $\IS^2$.
        Measures are discretized via their average on the subsets $m^k$.
        Functions are discretized on the points $z^k$ (dot markers), their
        gradients are discretized on the $y^j$ (square markers).
        Gradients are computed from points in a neighborhood $\nbhd_j$ of
        $y^j$.
        The neighborhood relation is depicted with dashed lines.       
        The discretization points were obtained by recursively subdividing the
        $20$ triangular faces of an icosahedron and projecting the vertices to
        the surface of the sphere after each subdivision.
    }\label{fig:discr-sphere}
\end{figure}
We assume a $d$-dimensional image domain $\Omega$, $d = 2,3$, that is
discretized using $n$ points $x^1, \dots, x^n \in \Omega$.
Differentiation in $\Omega$ is done on a staggered grid with Neumann boundary
conditions such that the dual operator to the differential operator $D$ is the
negative divergence with vanishing boundary values.

The framework presented in Section \ref{sec:theory} applies to arbitrary
compact metric spaces $X$.
However, for an efficient implementation of the Lipschitz constraint in
\eqref{eq:measure-valued-tv}, we will assume an $s$-dimensional manifold
$X = \IM$.
This includes the case of ODF-valued images ($X = \IM = \IS^2$, $s=2$).
For future generalizations to other manifolds, we give the discretization in
terms of a general manifold $X = \IM$ even though this means neglecting the
reasonable parametrization of $\IS^2$ using spherical harmonics in the case
of DW-MRI.
Moreover, note that the following discretization does not apply to arbitrary
metric spaces $X$.

Now, let $\IM$ be decomposed (Fig.~\ref{fig:discr-sphere}) into $l$ disjoint
measurable (not necessarily open or closed) sets
\begin{equation}
    m^1, \dots, m^l \subset \IM
\end{equation}
with $\bigcup_k m^k = \IM$ and volumes $b^1, \dots, b^l \in \R$ with respect to
the Lebesgue measure on $\IM$.
A measure-valued function $u\colon \Omega \to \IP(\IM)$ is discretized as its
average $u \in \R^{n,l}$ on the volume $m^k$, i.e.,
\begin{equation}
    u_k^i := u_{x^i}(m^k)/b_{k}.
\end{equation}

Functions $p \in C_c^1(\Omega, \Lip(X,\R^d))$ as they appear for example in our
proposed formulation of $\TV$ in \eqref{eq:bv-valued-tv} are identified with
functions $p\colon \Omega \times \IM \to \R^d$ and discretized as
$p \in \R^{n,l,d}$ via $p_{kt}^i := p_t(x^i, z^k)$ for a fixed choice of
discretization points
\begin{equation}
    \forall k=1,\dots,l: \quad z^k \in m^k \subset \IM.
\end{equation}
The dual pairing of $p$ with $u$ is discretized as
\begin{equation}\label{eq:discrete-int}
    \langle u, p \rangle_b := \sum_{i,k} b_{k} u_k^i p_k^i.
    \vspace{-7.5pt}
\end{equation}

\subsubsection{Implementation of the Lipschitz Constraint}
The Lipschitz constraint in the definition \eqref{eq:krdual} of $W_1$ and in
the definition \eqref{eq:measure-valued-tv} of $\TV_{\KR}$ is implemented as a
norm constraint on the gradient.
Namely, for a function $p\colon \IM \to \R$, which we discretize as $p \in \R^{l}$,
$p_k := p(z^k)$, we discretize gradients on a staggered grid of $m$ points
\begin{equation}
    y^1, \dots, y^m \in \IM,
\end{equation}
such that each of the $y^j$ has $r$ neighboring points among the $z^k$
(Fig.~\ref{fig:discr-sphere}):
\begin{equation}
    \forall j=1,\dots,m: \quad
        \nbhd_j \subset \{1, \dots, l\},\quad \#\nbhd_j = r.
\end{equation}
The gradient $g \in \R^{m,s}$, $g^j := Dp(y^j)$, is then defined as the vector
in the tangent space at $y^j$ that, together with a suitable choice of the
unknown value $c := p(y^j)$, best explains the known values of $p$ at the
$z^k$ by a first-order Taylor expansion
\begin{equation}
    p(z^k) \approx p(y^j) + \langle g^j, v^{jk} \rangle,
        \quad k \in \nbhd_j,
\end{equation}
where $v^{jk} := \exp^{-1}_{y^j}(z^k) \in T_{y^j}\IM$ is the Riemannian inverse
exponential mapping of the neighboring point $z^k$ to the tangent space at
$y^j$.
More precisely,
\begin{equation}\label{eq:argmin-gradient}
    g^j := \argmin_{g \in T_{y^j}\IM} \min_{c\in\R} \sum_{k \in \nbhd_j}
        \left(c + \langle g, v^{jk} \rangle - p(z^k)\right)^2.
\end{equation}
Writing the $v^{jk}$ into a matrix $M^j \in \R^{r,s}$ and encoding the
neighboring relations as a sparse indexing matrix $P^j \in \R^{r,l}$,
we obtain the explicit solution for the value $c$ and gradient $g^j$ at the
point $y^j$ from the first-order optimality conditions
of \eqref{eq:argmin-gradient}:
\begin{align}
    &c = p(y^j) = \frac{1}{r}(e^T P^j p - e^T M^j g^j), \\
    &(M^j)^T E M^j g^j = (M^j)^T E P^j p,
\end{align}
where $e := (1,\dots,1) \in \R^r$ and $E := (I - \frac{1}{r}ee^T)$.
The value $c$ does not appear in the linear equations for $g^j$ and is not
needed in our model, therefore we can ignore the first line.
The second line, with
$A^j := (M^j)^T E M^j \in \R^{s,s}$ and $B^j := (M^j)^T E \in \R^{s,r}$,
can be concisely written as
\begin{equation}
    A^j g^j = B^j P^j p, \text{ for each } j \in \{1, \dots, m \}.
\end{equation}
Following our discussion about the choice of norm in
Appendix~\ref{apdx:product-norms}, the (Lipschitz) norm constraint
$\|g_j\| \leq 1$ can be implemented using the Frobenius norm or the spectral
norm, both being rotationally invariant and both acting as desired on
cartoon-like jump functions (cf. Prop. \ref{prop:cartoon-tv}).
\subsubsection{Discretized $W_1$-$\TV$ Model}\label{sec:w1-tv}
Based on the above discretization, we can formulate saddle-point forms for
\eqref{eq:W1TV-functional} that allow to apply a primal-dual
first-order method such as \cite{chamb2011}.
In the following, the measure-valued input or reference image is given by
$f \in \R^{l,n}$ and the dimensions of the primal and dual variables are
\begin{align}
    & u \in \R^{l,n}, && p \in \R^{l,d,n}, && g \in \R^{n,m,s,d}, \\
    & p_0 \in \R^{l,n}, && g_0 \in \R^{n,m,s},
\end{align}
where $g^{ij} \approx D_z p(x^i, y^j)$ and $g_0^{j} \approx D p_0(y^j)$.

Using a $W_1$ data term, the saddle point form of the overall problem reads
\begin{align}
    \min_{u} \max_{p,g} \quad
        & W_1(u,f) + \langle Du, p \rangle_b \\
    \text{s.t.}\quad
        & u^i \geq 0, ~\langle u^i, b \rangle = 1, ~\forall i, \\
        & A^j g^{ij}_t = B^j P^j p^i_t ~\forall i,j,t, \\
        & \|g^{ij}\| \leq \lambda ~\forall i,j
\end{align}
or, applying the Kantorovich-Rubinstein duality \eqref{eq:krdual} to the
data term,
\begin{align}\label{eq:W1TV-saddle-point}
    \min_{u} \max_{p, g, p_0, g_0} \quad
        & \langle u-f, p_0 \rangle_b + \langle Du, p \rangle_b \\
    \text{s.t.}\quad
        & u^i \geq 0, ~\langle u^i, b \rangle = 1 ~\forall i, \\
        & A^j g^{ij}_t = B^j P^j p^i_t,
          ~\|g^{ij}\| \leq \lambda ~\forall i,j,t, \\
        & A^j g^{ij}_0 = B^j P^j p^i_0,
          ~\|g^{ij}_0\| \leq 1 ~\forall i,j.
\end{align}
\subsubsection{Discretized $L^2$-$\TV$ Model}\label{sec:l2-tv}
For comparison, we also implemented the Rudin-Osher-Fatemi (ROF) model
\begin{equation}\label{eq:QBROF}
    \inf_{u:\Omega \to \IP(\IS^2)}
        \int_\Omega \int_{\IS^2} (f_x(z) - u_x(z))^2 \dd z \dd x
        + \lambda \TV(u)
\end{equation}
using $\TV=\TV_{\KR}$.
The quadratic data term can be implemented using the saddle point form
\begin{align}\label{eq:L2TV-saddle-point}
    \min_{u} \max_{p,g} \quad
        & \langle u-f, u-f \rangle_b
            + \langle Du, p \rangle_b \\
    \text{s.t.}\quad
        & u^i \geq 0, ~\langle u^i, b \rangle = 1, \\
        & A^j g^{ij}_t = B^j P^j p^i_t,
          ~\|g^{ij}\| \leq \lambda ~\forall i,j,t.
\end{align}
From a functional-analytic viewpoint, this approach requires to assume that
$u_x$ can be represented by an $L^2$ density, suffers from well-posedness
issues, and ignores the metric structure on $\IS^2$ as mentioned in the
introduction.
Nevertheless we include it for comparison, as the $L^2$ norm is a common choice
and the discretized model is a straightforward modification of the
$W_1$-$\TV$ model. %
\subsection{Implementation Using a Primal-Dual Algorithm}
Based on the saddle-point forms \eqref{eq:W1TV-saddle-point} and
\eqref{eq:L2TV-saddle-point}, we applied the primal-dual first-order method
proposed in \cite{chamb2011} with the adaptive step sizes from
\cite{goldstein2015a}.
We also evaluated the diagonal preconditioning proposed in \cite{pock2011}.
However, we found that while it led to rapid convergence in some cases, the
method frequently became unacceptably slow before reaching the desired accuracy.
The adaptive step size strategy exhibited a more robust overall convergence.

The equality constraints in \eqref{eq:W1TV-saddle-point} and
\eqref{eq:L2TV-saddle-point} were included into the objective function by
introducing suitable Lagrange multipliers.
As far as the norm constraint on $g_0$ is concerned, the spectral and Frobenius
norms agree, since the gradient of $p_0$ is one-dimensional.
For the norm constraint on the Jacobian $g$ of $p$, we found the spectral and
Frobenius norm to give visually indistinguishable results.

Furthermore, since $\IM = \IS^2$ and therefore $s=2$ in the ODF-valued case,
explicit formulas for the orthogonal projections on the spectral norm balls that
appear in the proximal steps are available~\cite{goldluecke2012natural}.
The experiments below were calculated using spectral norm constraints, as in our
experience this choice led to slightly faster convergence.
\section{Results}\label{sec:results}
We implemented our model in Python~3.5 using the libraries NumPy~1.13,
PyCUDA~2017.1 and CUDA~8.0.
The examples were computed on an Intel Xeon X5670 2.93GHz with 24 GB of main
memory and an NVIDIA GeForce GTX 480 graphics card
with 1,5 GB of dedicated video memory.
For each step in the primal-dual algorithm, a set of kernels was launched on the
GPU, while the primal-dual gap was computed and termination criteria were
tested every $5\,000$ iterations on the CPU.

For the following experiments, we applied our models presented in Sections
\ref{sec:w1-tv} ($W_1$-$\TV$) and \ref{sec:l2-tv} ($L_2$-$\TV$) to ODF-valued
images reconstructed from HARDI data using the reconstruction methods that are
provided by the Dipy project~\cite{garyf2014}:
\begin{itemize}
\item For voxel-wise QBI reconstruction within constant solid angle (CSA-ODF)
    \cite{aganj2009}, we used \texttt{CsaOdfModel} from \texttt{dipy.reconst.shm}
    with spherical harmonics functions up to order $6$.
\item We used the implementation \texttt{ConstrainedSphericalDeconvModel} as
    provided with \texttt{dipy.reconst.csdeconv} for voxel-wise CSD
    reconstruction as proposed in \cite{tournier2007}.
\end{itemize}
The response function that is needed for CSD reconstruction was determined
using the recursive calibration method \cite{tax2014} as implemented in
\texttt{recursive\_response}, which is also part of
\texttt{dipy.reconst.csdeconv}.
We generated the ODF plots using VTK-based \texttt{sphere\_funcs}
from \texttt{dipy.viz.fvtk}.

It is equally possibly to use other methods for Q-ball reconstruction for the
preprocessing step, or even integrate the proposed $\TV$-regularizer directly
into the reconstruction process.
Furthermore, our method is compatible with different numerical representations
of ODFs, including sphere discretization \cite{goh2009}, spherical
harmonics \cite{aganj2009}, spherical wavelets \cite{kezele2008}, ridgelets
\cite{michailovich2008} or similar basis functions \cite{kaden2011,ahrens2013},
as it does not make any assumptions on regularity or symmetry of the ODFs.
We leave a comprehensive benchmark to future work, as the main goal of this
work is to investigate the mathematical foundations.
\subsection{Synthetic Data}
\subsubsection{$L^2$-$\TV$ vs. $W_1$-$\TV$}
\begin{figure*}\centering
    \includegraphics[draft=false,
        trim=0 645 110 755,clip,width=\textwidth
    ]{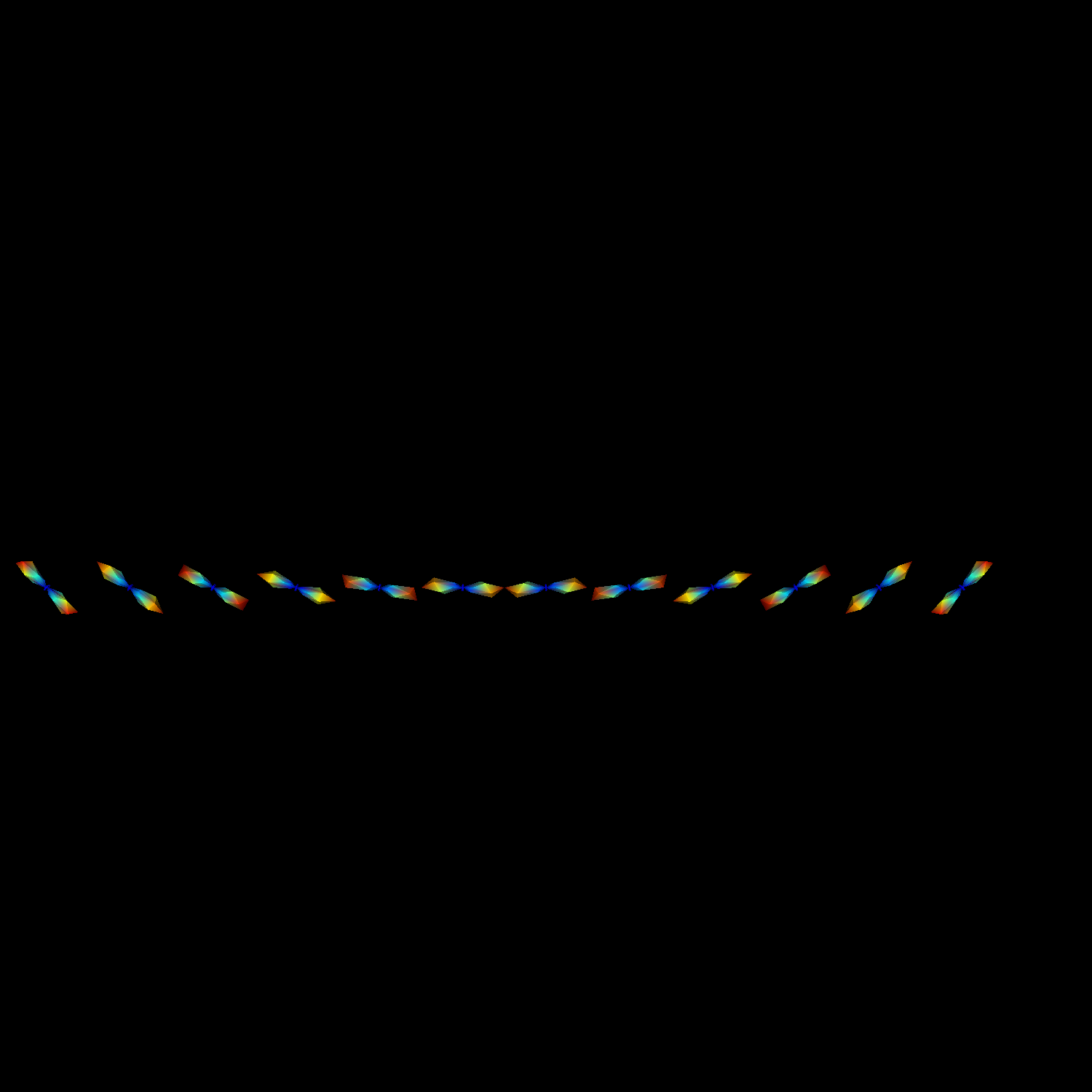}\\
    \includegraphics[draft=false,
        trim=0 645 110 755,clip,width=\textwidth
    ]{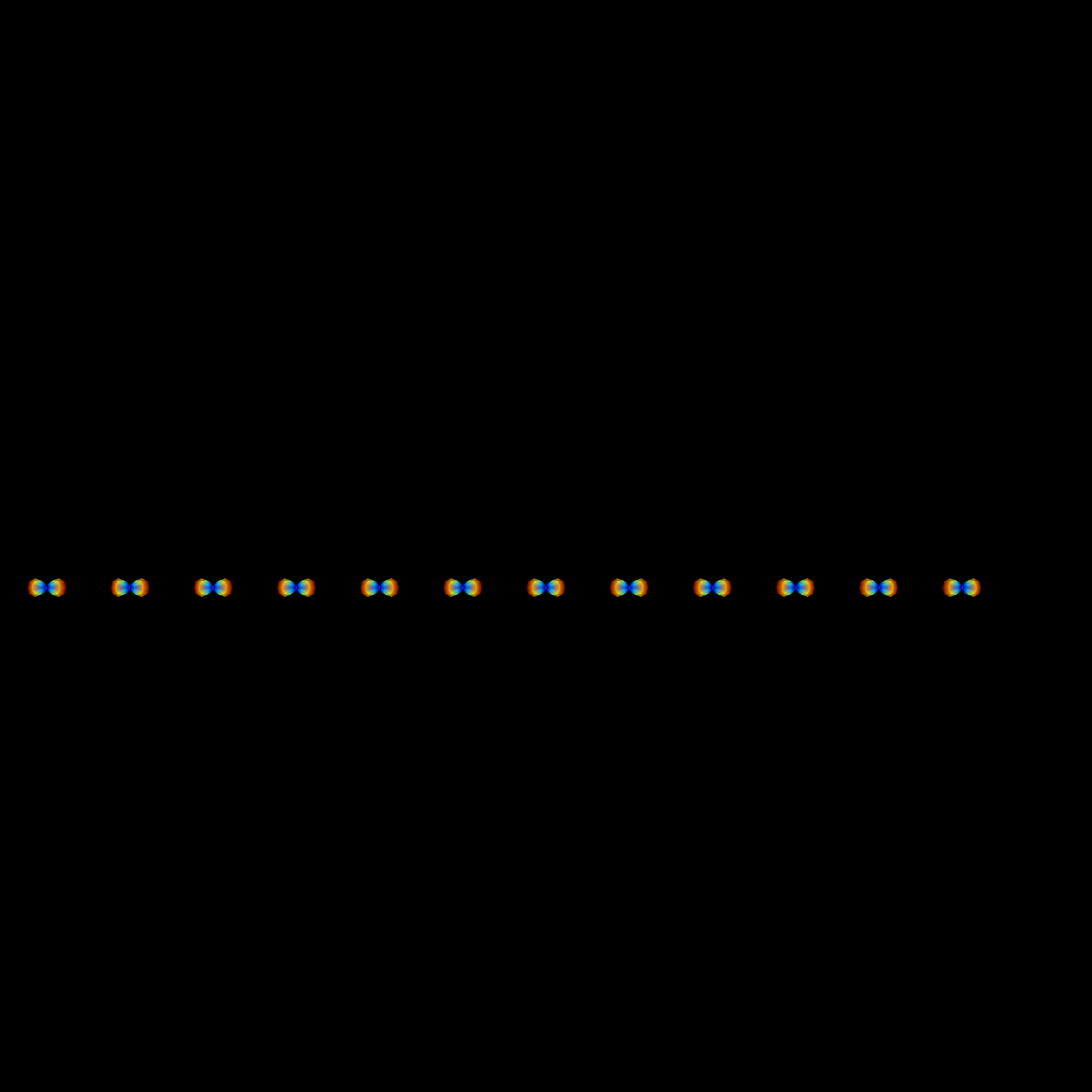}\\
    \includegraphics[draft=false,
        trim=0 645 110 755,clip,width=\textwidth
    ]{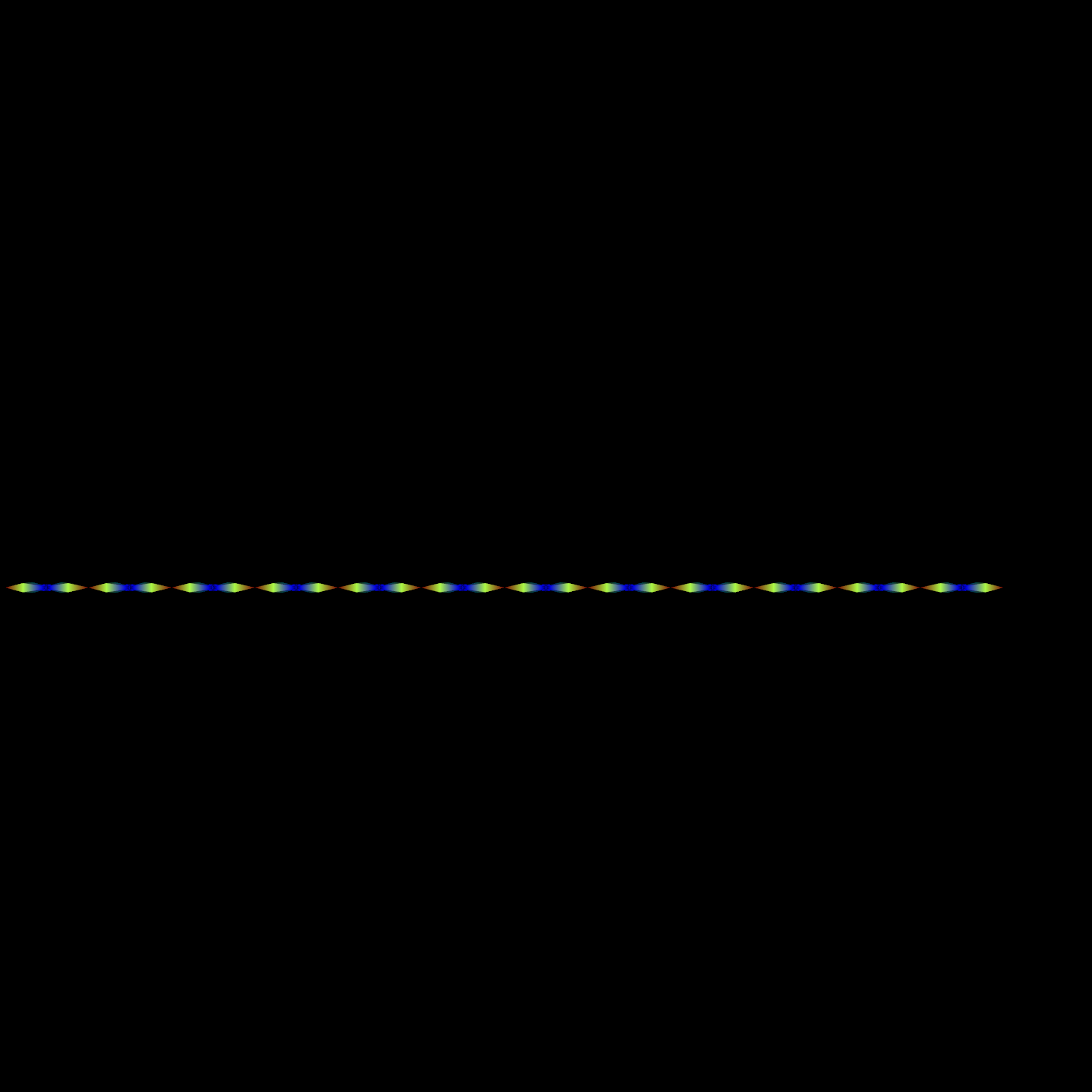}
    \caption{%
        \emph{Top:} 1D image of synthetic unimodal ODFs where the angle of the
            main diffusion direction varies linearly from left to right.
            This is used as input image for the center and bottom row.
        \emph{Center:} Solution of $L^2$-$\TV$ model with $\lambda=5$.
        \emph{Bottom:} Solution of $W_1$-$\TV$ model with $\lambda=10$.
        In both cases, the regularization parameter $\lambda$ was chosen
        sufficiently large to enforce a constant result.
        The quadratic data term mixes all diffusion directions into one blurred
        ODF, whereas the Wasserstein data term produces a tight ODF that is
        concentrated close to the median diffusion direction.
    }\label{fig:plot-oned}
\end{figure*} %
We demonstrate the different
behaviors of the $L^2$-$\TV$ model compared to the $W_1$-$\TV$ model with the
help of a one-dimensional synthetic image (Fig. \ref{fig:plot-oned}) generated using the
multi-tensor simulation method \texttt{multi\_tensor} from the module
\texttt{dipy.sims.voxel} which is based on \cite{stejskal1965} and
\cite[p.~42]{descoteaux2008}; see also \cite{Vogt2017}.

By choosing very high regularization parameters $\lambda$, we enforce
the models to produce constant results.
The $L^2$-based data term prefers a blurred mixture of diffusion directions,
essentially averaging the probability measures.
The $W_1$ data term tends to concentrate the mass close to the median of the
diffusion directions on the unit sphere, properly taking into account the metric
structure of $\IS^2$. %
\subsubsection{Scale-space Behavior}\label{sec:fiber-phantom}
To demonstrate the scale space behavior of our variational models, we
implemented a 2-D phantom of two crossing fibre bundles as depicted in
Fig.~\ref{fig:phantom-plot}, inspired by~\cite{ouyang2014}.
From this phantom we computed the peak directions of fiber orientations on a
$15 \times 15$ grid.
This was used to generate synthetic HARDI data simulating a DW-MRI
measurement with $162$ gradients and a $b$-value of $3\,000$, again using the
multi-tensor simulation framework from \texttt{dipy.sims.voxel}.

We then applied our models to the CSA-ODF reconstruction of this data set for
increasing values of the regularization parameter $\lambda$ in order to
demonstrate the scale-space behaviors of the different data terms
(Fig.~\ref{fig:phantom-scalespace}).

As both models use the proposed $\TV$ regularizer, edges are preserved.
However, just as classical ROF models tend to reduce jump sizes across edges,
and lose contrast, the $L^2$-$\TV$ model results in the background and
foreground regions becoming gradually more similar as regularization strength
increases.
The $W_1$-$\TV$ model preserves the unimodal ODFs in the background regions 
and demonstrates a behavior more akin to robust $L^1$-$\TV$ models
\cite{duval2009}, with structures disappearing abruptly  rather than gradually
depending on their scale.
\subsubsection{Denoising}
We applied our model to the CSA-ODF reconstruction of a slice (NumPy coordinates
\texttt{[12:27,22,21:36]}) from the synthetic HARDI data set with added noise at
$\SNR=10$, provided in the ISBI 2013 HARDI reconstruction challenge.
We evaluated the angular precision of the estimated fiber compartments using the
script (\texttt{compute\_local\_metrics.py}) provided on the challenge
homepage~\cite{isbi2013}.

The script computes the mean $\mu$ and standard deviation $\sigma$ of the
angular error between the estimated fiber directions inside the voxels and the
ground truth as also provided on the challenge page
(Fig.~\ref{fig:isbi-groundtruth}).

The noisy input image exhibits a mean angular error of $\mu = 34.52$ degrees
($\sigma = 19.00$).
The reconstructions using $W_1$-$\TV$ ($\mu = 17.73$, $\sigma = 17.25$) and
$L^2$-$\TV$ ($\mu = 17.82$, $\sigma = 18.79$) clearly improve the angular error
and give visually convincing results:
The noise is effectively reduced and a clear trace of fibres becomes visible
(Fig.~\ref{fig:isbi-results}).
In these experiments, the regularizing parameter $\lambda$ was chosen optimally
in order to minimize the mean angular error to the ground truth.
\subsection{Human Brain HARDI Data}
One slice (NumPy coordinates \texttt{[20:50, 55:85, 38]}) of HARDI data from
the human brain data set \cite{rokem2013} was used to demonstrate the
applicability of our method to real-world problems and to images reconstructed
using CSD (Fig.~\ref{fig:realworld-dataset}).
Run times of the $W_1$-$\TV$ and $L^2$-$\TV$ model are approximately 35 minutes
($10^5$ iterations) and 20 minutes ($6\cdot 10^4$ iterations).

As a stopping criterion, we require the primal-dual gap to fall below $10^{-5}$,
which corresponds to a deviation from the global minimum of less than
$0.001 \%$, and is a rather challenging precision for the first-order methods
used.
The regularization parameter $\lambda$ was manually chosen based on visual
inspection.

Overall, contrast between regions of isotropic and
anisotropic diffusion is enhanced.
In regions where a clear diffusion direction is already visible before spatial
regularization, $W_1$-$\TV$ tends to conserve this information better than
$L^2$-$\TV$.
\section{Conclusion and Outlook}
Our mathematical framework for ODF- and, more general, measure-valued images
allows to perform total vari\-a\-tion-based regularization of measure-valued
data without assuming a specific parametrization of ODFs,
while correctly taking the metric on $\IS^2$ into account.
The proposed model penalizes jumps in cartoon-like images proportional to the
jump size measured on the underlying normed space, in our case the
Kan\-toro\-vich-Rubin\-stein space, which is built on the Wasserstein-1-metric.
Moreover, the full variational problem was shown to have a solution and can be
implemented using off-the-shelf numerical methods.

With the first-order primal-dual algorithm chosen in this paper, solving the
underlying optimization problem for DW-MRI regularization is computationally
demanding due to the high dimensionality of the problem.
However, numerical performance was not a priority in this work and can
be improved.
For example, optimal transport norms are known to be efficiently computable
using Sinkhorn's algorithm \cite{cuturi2013}.

A particularly interesting direction for future research concerns extending the approach to
simultaneous reconstruction and regularization, with an additional (non-) linear
operator in the data fidelity term~\cite{aganj2009}.
For example, one could  consider an integrand of the form $\rho(x,u(x)) := d(S(x),Au(x))$ for some
measurements $S$ on a metric space $(H,d)$ and a forward operator $A$ mapping
an ODF $u(x) \in \IP(\IS^2)$ to~$H$.

Furthermore, modifications of our total variation seminorm that take into account
the coupling of positions and orientations according to the physical
interpretation of ODFs in DW-MRI could close the gap to state-of-the-art
approaches such as \cite{duits2011,portegies2017}.

The model does not require symmetry of the ODFs, and therefore
could be adapted to novel asymmetric ODF
approaches~\cite{delp2007,ehric2011,reis2012,CetinKarayumak2018}.
Finally, it is easily extendable to images with values in the
probability space over a different manifold, or even a metric space, as they
appear for example in statistical models of computer vision~\cite{sriv2007}
and in recent lifting approaches \cite{mollenhoff2016,laude2016,astrom2017}
for combinatorial and non-convex optimization problems.
\begin{figure*}[p]
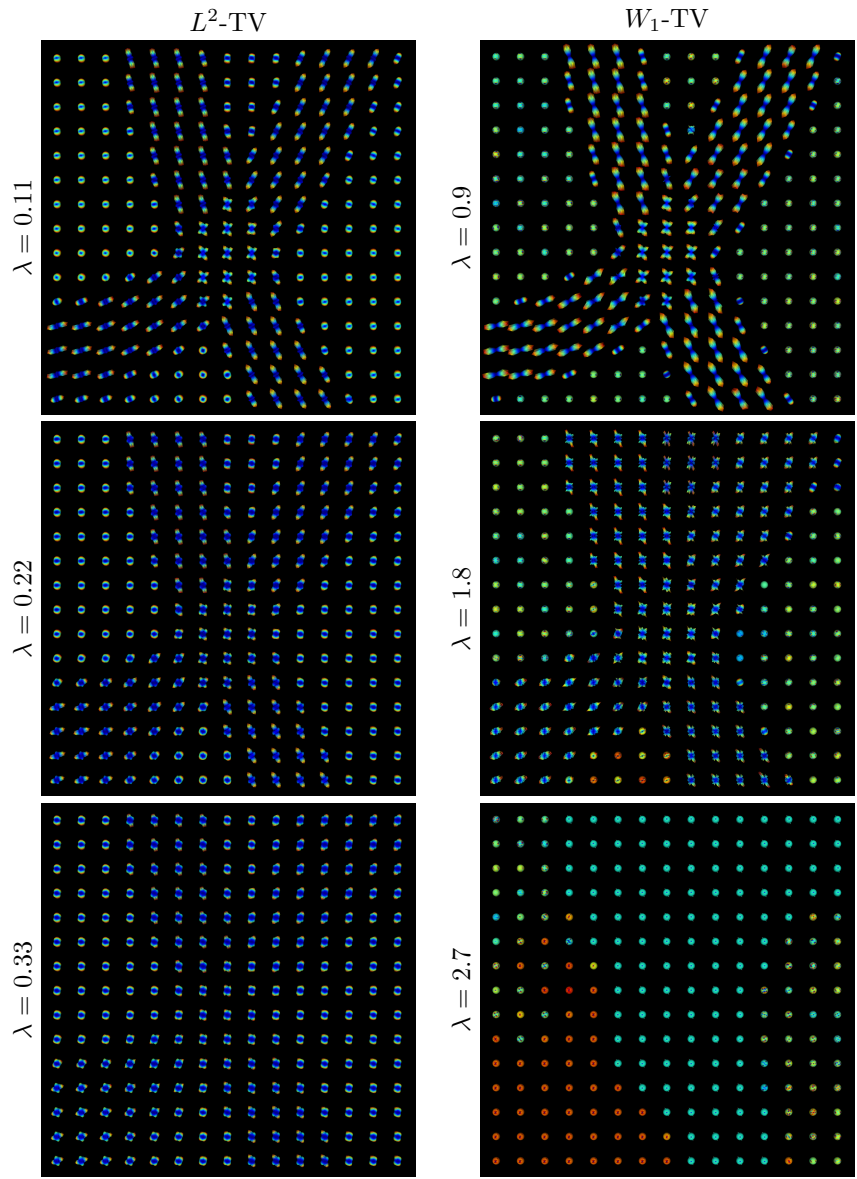
\centering
    \begin{tikzpicture}
        \newlength\picwidth
        \setlength\picwidth{0.41\textwidth}
        \newlength\picx
        \setlength\picx{1.02\picwidth}
        \newlength\picy
        \setlength\picy{\picx}
        
        \def\quadlist{0.11,0.22,0.33}
        \def\wasslist{0.9,1.8,2.7}
        
        \foreach \quadlbd [count=\i,%
                           evaluate=\i as \wasslbd using {{\wasslist}[\i-1]}]%
                 in \quadlist {
                 
            \node (quad\i) at (0,-\i\picy) {
                \includegraphics[%
                    trim=0 0 85 85,clip,width=\picwidth
                ]{fig/plot-phantom-quadratic-\i}
            };
            \node[left=0cm of quad\i,inner sep=0pt,anchor=east]
                {\rotatebox{90}{$\lambda = \quadlbd$}};
            
            \node (wass\i) at (1.15\picx,-\i\picy) {
                \includegraphics[%
                    trim=0 0 90 90,clip,width=\picwidth
                ]{fig/plot-phantom-W1-\i}
            };
            \node[left=0cm of wass\i,inner sep=0pt,anchor=east]
                {\rotatebox{90}{$\lambda = \wasslbd$}};
            
        }
        
        \node[above=0cm of quad1,inner sep=0pt,anchor=south] {$L^2$-$\TV$};
        \node[above=0cm of wass1,inner sep=0pt,anchor=south] {$W_1$-$\TV$};
    \end{tikzpicture}
    \caption{%
        Numerical solutions of the proposed variational models (see Sections
        \ref{sec:w1-tv} and \ref{sec:l2-tv}) applied to the
        phantom (Fig.~\ref{fig:phantom-plot}) for increasing
        values of the regularization parameter $\lambda$. %
        \emph{Left column:} Solutions of $L^2$-$\TV$ model for
            $\lambda = 0.11,\,0.22,\,0.33$.
        \emph{Right column:} Solutions of $W_1$-$\TV$ model for
            $\lambda = 0.9,\,1.8,\,2.7$.
        As is known from classical ROF models, the $L^2$ data term produces a
        gradual transition/loss of contrast towards the constant image,
        while the $W_1$ data term stabilizes contrast along the edges.
    }\label{fig:phantom-scalespace}
\end{figure*} %
\setlength\picwidth{0.68\textwidth}
\begin{figure*}[p]\centering
    \includegraphics[%
        trim=8 8 8 8,clip,width=\picwidth
    ]{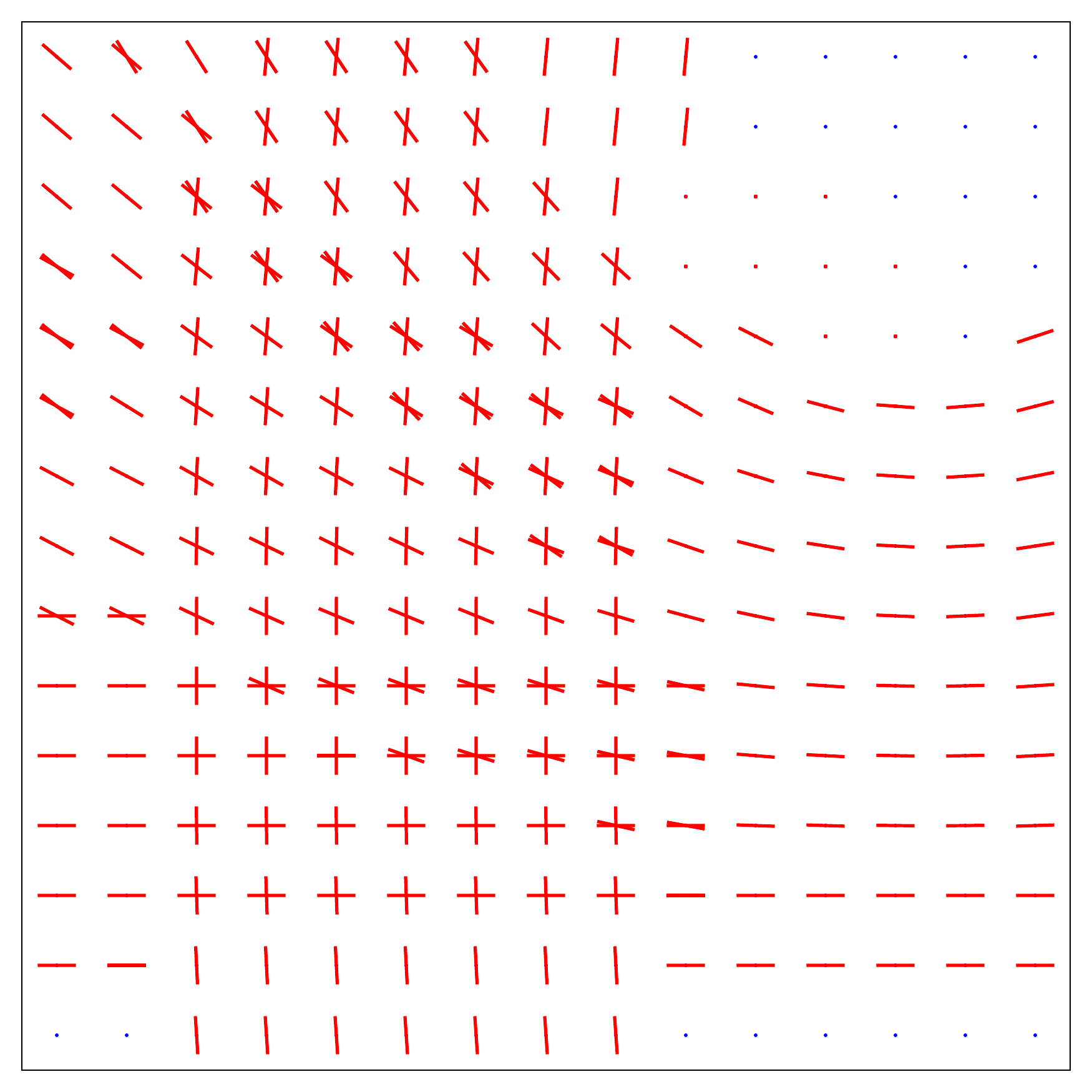}
    \includegraphics[%
        trim=0 0 90 90,clip,width=\picwidth
    ]{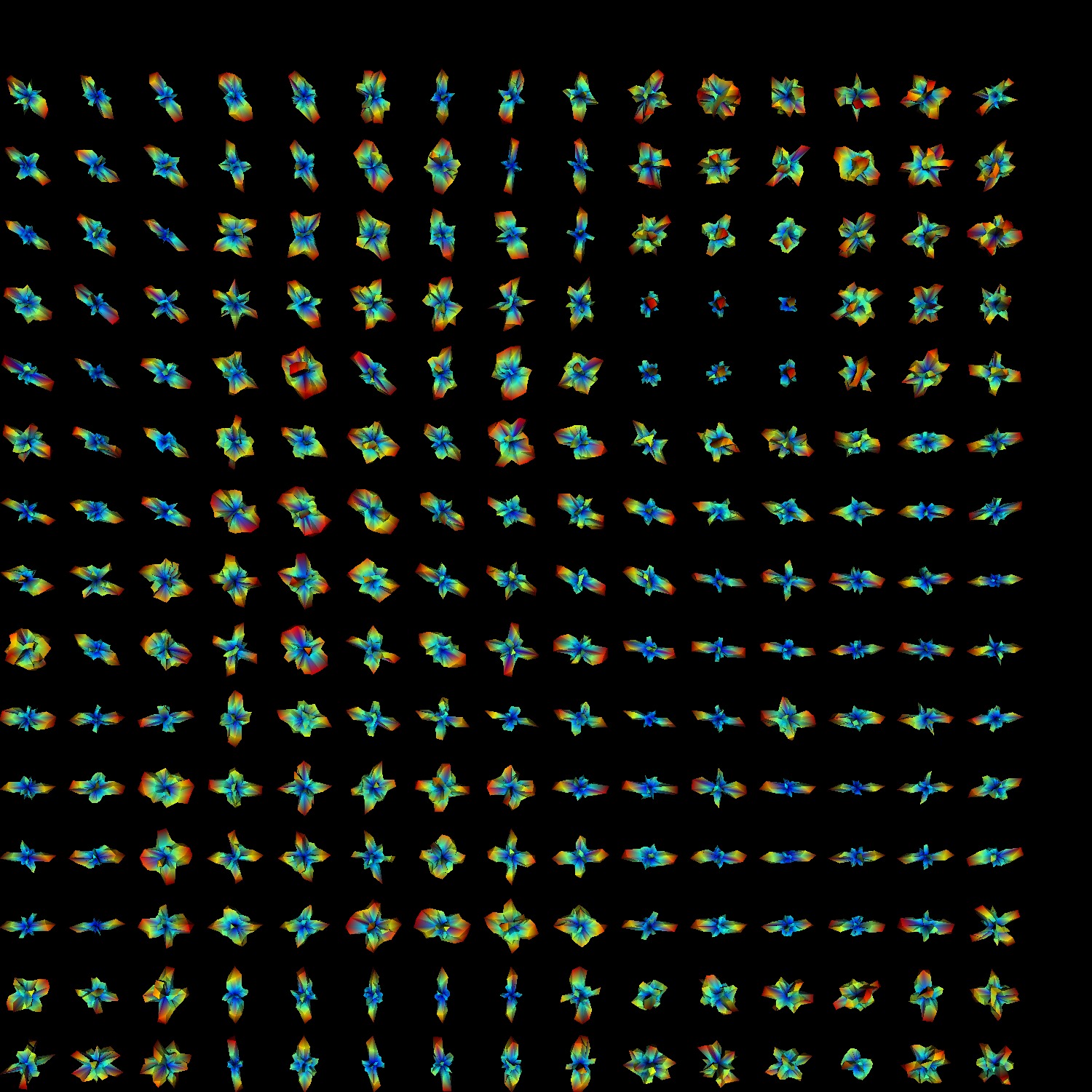}
    \caption{%
        Slice of size $15 \times 15$ from the data provided for the ISBI 2013
        HARDI reconstruction challenge~\protect\cite{isbi2013}.
        \emph{Left:} Peak directions of the ground truth.
        \emph{Right:} Q-ball image reconstructed from the noisy ($\SNR=10$)
            synthetic HARDI data, without spatial regularization.
        The low $\SNR$ makes it hard to visually recognize the fiber directions.
    }\label{fig:isbi-groundtruth}
\end{figure*}
\begin{figure*}[p]\centering\vspace{0.5cm}
    \includegraphics[%
        trim=0 0 90 90,clip,width=\picwidth
    ]{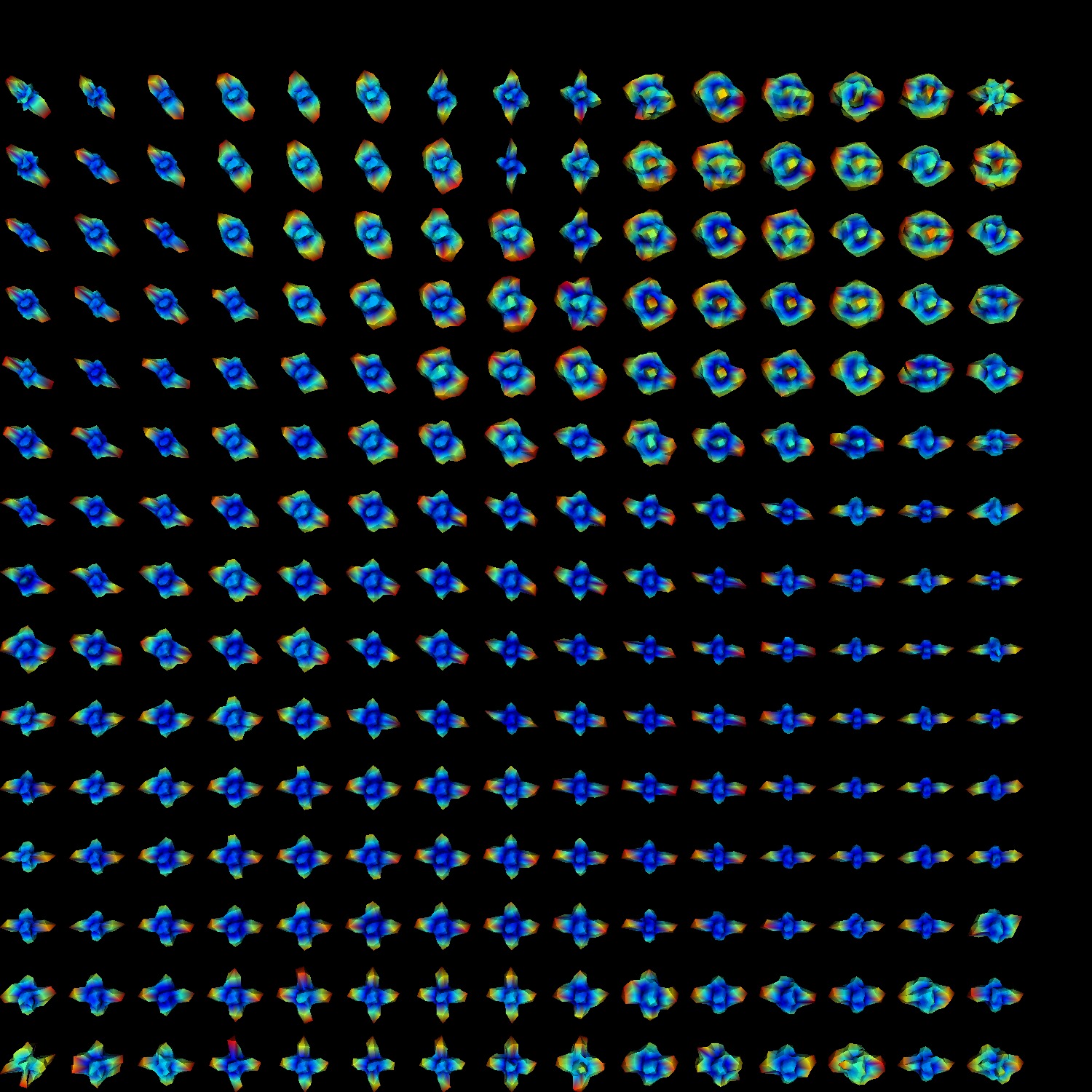}
    \includegraphics[%
        trim=0 0 90 90,clip,width=\picwidth
    ]{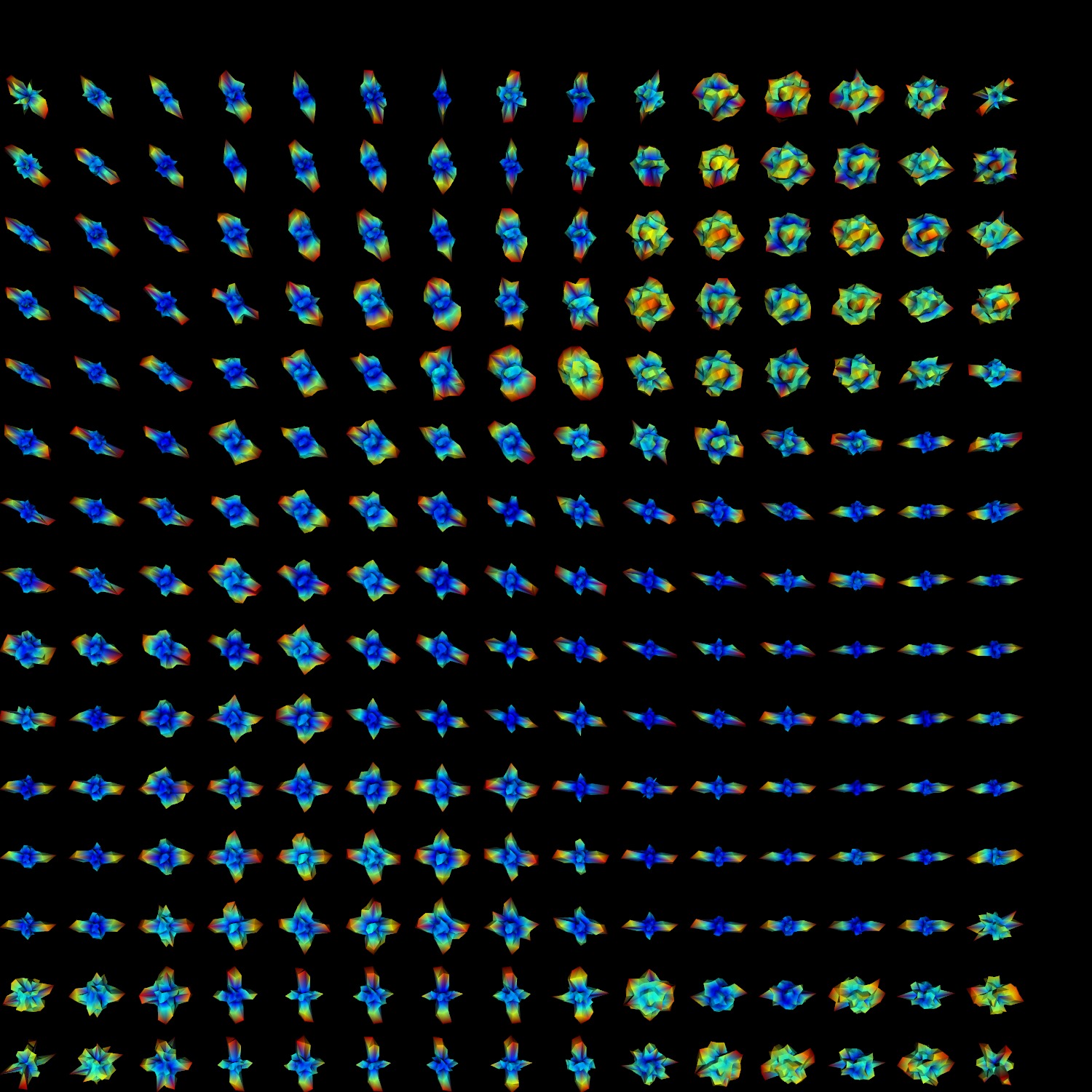}
    \caption{%
        Restored Q-ball images reconstructed from the noisy input data in
        Fig.~{\ref{fig:isbi-groundtruth}}.
        \emph{Left:} Result of the $L^2$-$\TV$ model ($\lambda=0.3$).
        \emph{Right:} Result of the $W_1$-$\TV$ model ($\lambda=1.1$).
        The noise is reduced substantially so that fiber traces are clearly
        visible in both cases.
        The \mbox{$W_1$-$\TV$} model generates less diffuse distributions.
    }\label{fig:isbi-results}
\end{figure*} %
\begin{figure}[p]\centering
    \setlength\picwidth{0.46\textwidth}
    \newlength\spywidth
    \setlength\spywidth{0.4\picwidth}
    \begin{tikzpicture}[spy using outlines={%
            height = \spywidth, width = \spywidth,
            magnification = 3, orange, connect spies
    }]
        \node[inner sep=0] (rw-csd) at (0,0) {
            \includegraphics[%
                trim=10 10 57 57,clip,width=\picwidth
            ]{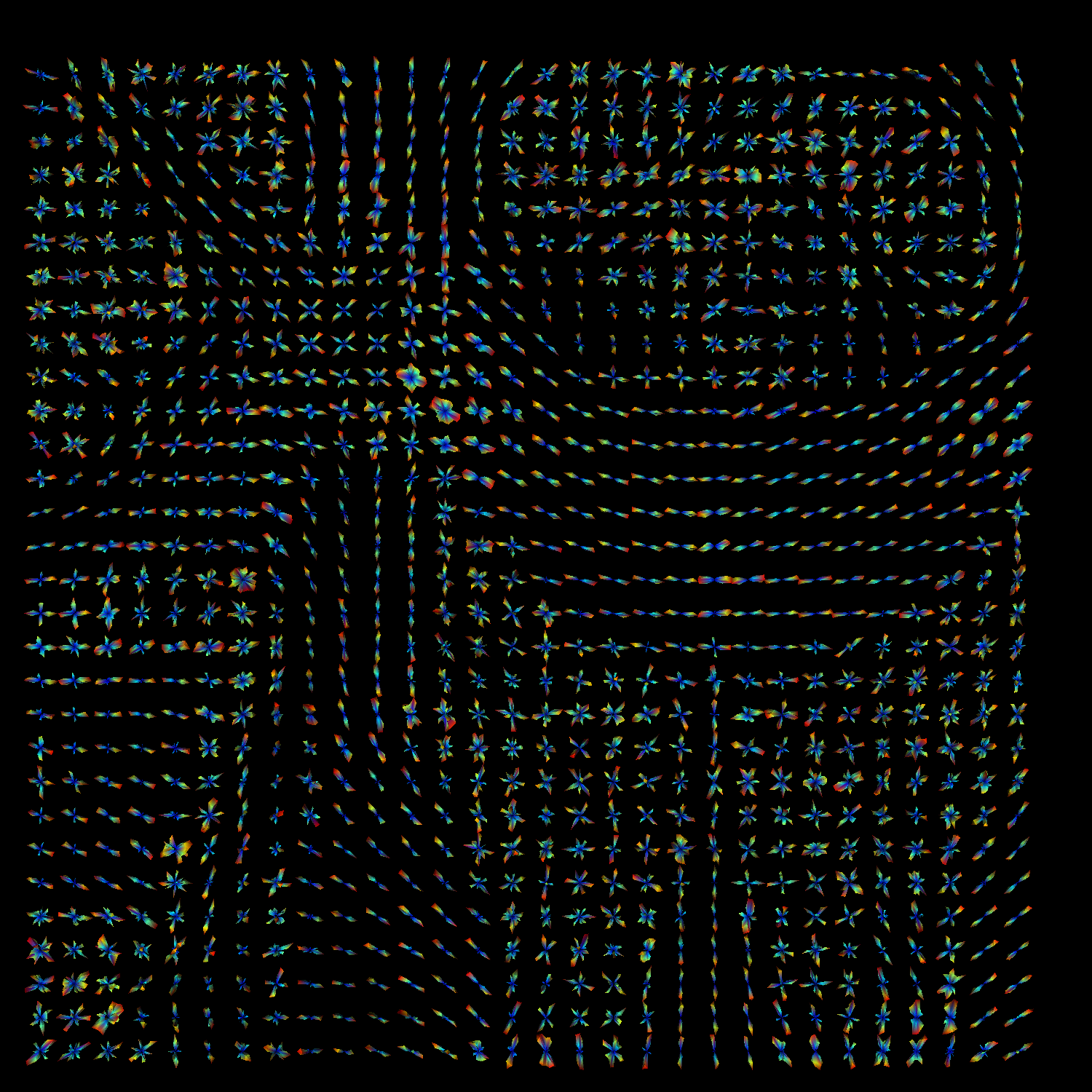}
        };
        \RelativeSpy{rw-spy1}{rw-csd}{(.76,.886)}{(1.07,0.73)}
        \RelativeSpy{rw-spy2}{rw-csd}{(.275,.532)}{(1.07,0.265)}
        
        \node[inner sep=0,below=0.1cm of rw-csd] (rw-csd-quadratic) {%
            \includegraphics[%
                trim=10 10 57 57,clip,width=\picwidth
            ]{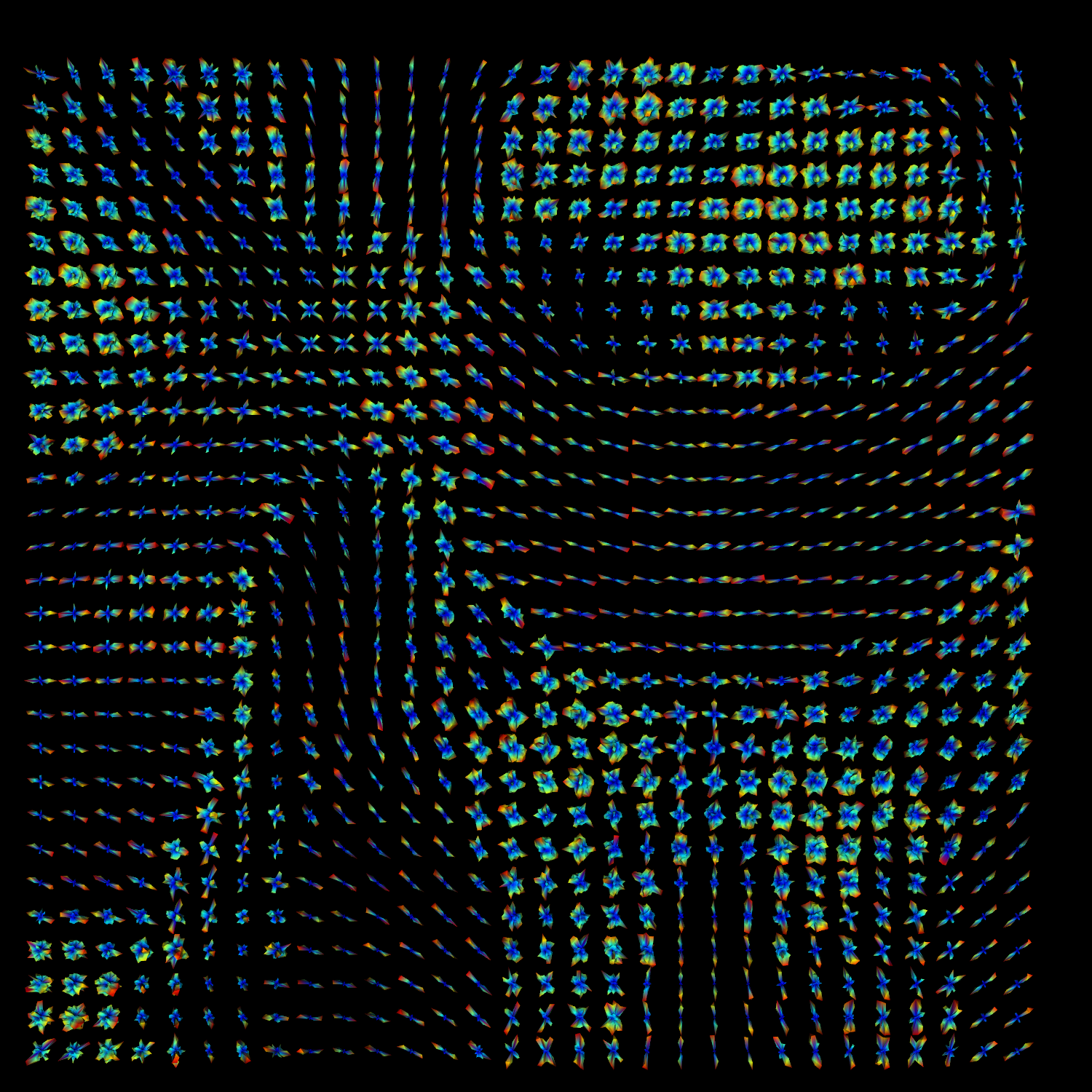}
        };
        \RelativeSpy{rw-spy3}{rw-csd-quadratic}{(.76,.886)}{(1.07,0.73)}
        \RelativeSpy{rw-spy4}{rw-csd-quadratic}{(.275,.532)}{(1.07,0.265)}
        
        \node[inner sep=0,below=0.1cm of rw-csd-quadratic] (rw-csd-W1) {%
            \includegraphics[%
                trim=10 10 57 57,clip,width=\picwidth
            ]{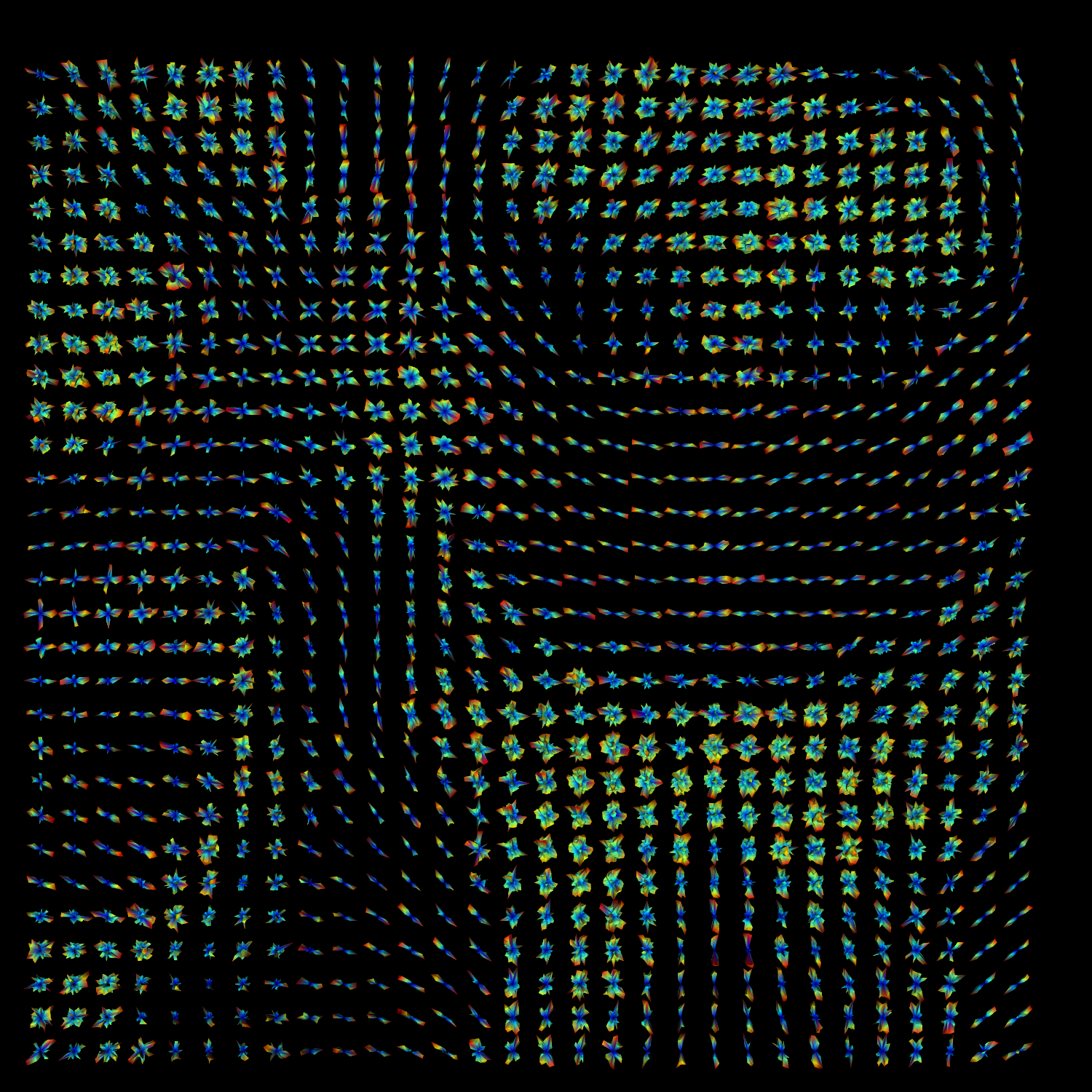}
        };
        \RelativeSpy{rw-spy5}{rw-csd-W1}{(.76,.886)}{(1.07,0.73)}
        \RelativeSpy{rw-spy6}{rw-csd-W1}{(.275,.532)}{(1.07,0.265)}
    \end{tikzpicture}
    \caption{%
        ODF image of the corpus callosum, reconstructed with CSD from HARDI data
        of the human brain \protect\cite{rokem2013}.
        \emph{Top:} Noisy input.
        \emph{Middle:} Restored using $L^2$-$\TV$ model ($\lambda=0.6$).
        \emph{Bottom:} Restored using $W_1$-$\TV$ model ($\lambda=1.1$).
        The results do not show much difference: 
        Both models enhance contrast between regions of isotropic and anisotropic
        diffusion while the anisotropy of ODFs is conserved.
    }\label{fig:realworld-dataset}
\end{figure} %
\appendix\numberwithin{equation}{section}\normalsize
\section{Background from Functional Analysis and \\
         Measure Theory}\label{apdx:background}
In this appendix, we present the theoretical background for a rigorous
understanding of the notation and definitions underlying the notion of $\TV$ as
proposed in \eqref{eq:bv-valued-tv} and \eqref{eq:measure-valued-tv}.
Subsection \ref{apdx:bv-bv} is concerned with Banach-space valued functions and
subsection \ref{apdx:kr} focuses on the special case of measure-valued functions.
\subsection{Banach Space-Valued Functions of Bounded Variation}\label{apdx:bv-bv}
This subsection introduces a function space on which the formulation of $\TV$ as
given in \eqref{eq:bv-valued-tv} is well-defined.

Let $(V, \|\cdot\|_V)$ be a real Banach space with (topological) dual space
$V^*$, i.e., $V^*$ is the set of bounded linear operators from $V$ to $\R$.
The dual pairing is denoted by $\langle p, v \rangle := p(v)$ whenever
$p \in V^*$ and $v \in V$.

We say that $u\colon \Omega \to V$ is \emph{weakly~measurable} if
$x \mapsto \langle p, u(x) \rangle$ is measurable for each $p \in V^*$ and say
that $u \in L_w^\infty(\Omega, V)$ if $u$ is weakly~measurable and essentially
bounded in $V$, i.e.,
\begin{equation}\label{eq:l-inf-w-norm}
    \|u\|_{\infty,V} := \esssup_{x \in \Omega} \|u(x)\|_V < \infty.
\end{equation}
Note that the essential supremum is well-defined even for non-measurable
functions as long as the measure is complete. In our case, we assume the
Lebesgue measure on $\Omega$ which is complete.

The following Lemma ensures that the integrand in \eqref{eq:bv-valued-tv} is
measurable.
\begin{lemma}\label{lem:measurability1}
    Assume that $u\colon \Omega \to V$ is weakly~measurable and
    $p\colon \Omega \to V^*$ is weakly*~continuous, i.e., for each $v \in V$,
    the map $x \mapsto \langle p(x), v \rangle$ is continuous.
    Then the map $x \mapsto \langle p(x), u(x) \rangle$ is measurable.
\end{lemma}
\begin{proof}
    Define $f\colon \Omega \times \Omega \to \R$ via
    \begin{equation}
        f(x, \xi) := \langle p(x), u(\xi) \rangle.
    \end{equation}
    Then $f$ is continuous in the first and measurable in the second variable.
    In the calculus of variations, functions with this property are called
    Carathéodory functions and have the property that $x \mapsto f(x,g(x))$ is
    measurable whenever $g\colon \Omega \to \Omega$ is
    measurable, which is proven by approximation of $g$ as the pointwise
    limit of simple functions~\cite[Prop. 3.7]{dacorogna2008}.
    In our case we can simply set $g(x) := x$, which is measurable, and the
    assertion follows.
\qed
\end{proof}
\subsection{Wasserstein Metrics and the KR Norm}\label{apdx:kr}
This subsection is concerned with the definition of the space of measures
$\KR(X)$ and the isometric embedding $\IP(X) \subset \KR(X)$ underlying the
formulation of $\TV$ given in \eqref{eq:measure-valued-tv}.

By $\IM(X)$ and $\IP(X) \subset \IM(X)$, we denote the sets of
signed Radon measures and Borel probability measures supported on $X$.
$\IM(X)$ is a vector space \cite[p.~360]{Hewitt1965} and a Banach space if
equipped with the norm
\begin{equation}
    \|\mu\|_\IM := \int_X \dd |\mu|,
\end{equation}
so that a function $u\colon \Omega \to \IP(X) \subset \IM(X)$ is Banach
space-valued (i.e., $u$ takes values in a Banach space).
If we define $C(X)$ as the space of continuous functions on $X$ with norm
$\|f\|_{C} := \sup_{x \in X} |f(x)|$, under the above assumptions on $X$,
$\IM(X)$ can be identified with the (topological) dual space of $C(X)$ with
dual pairing
\begin{equation}\label{eq:dual-pairing}
    \langle \mu, p \rangle := \int_{X} p \dd \mu,
\end{equation}
whenever $\mu \in \IM(X)$ and $p \in C(X)$, as proven in
\cite[p.~364]{Hewitt1965}.
Hence, $\IP(X)$ is a bounded subset of a dual space.

We will now see that additionally, $\IP(X)$ can be regarded as subset of a
Banach space which is a \emph{predual} space (in the sense that its dual space
can be identified with a ``meaningful'' function space) and which metrizes the
weak* topology of $\IM(X)$ on $\IP(X)$ by the optimal transport metrics we are
interested in.

For $q \geq 1$, the Wasserstein metrics $W_q$ on $\IP(X)$ are defined via
\begin{equation}
\begin{aligned}
    W_q(\mu, \mu') := &\left(\inf_{\gamma \in \Gamma(\mu,\mu')}
        \int_{X \times X}
             d(x,y)^q
        \dd \gamma(x,y) \right)^{1/q},
\end{aligned}
\end{equation}
where
\begin{equation}
    \Gamma(\mu,\mu') := \left\{ \gamma \in \IP(X \times X)\colon
        ~\proj_1 \gamma = \mu, ~\proj_2 \gamma = \mu'
    \right\}.
\end{equation}
Here, $\proj_i\gamma$ denotes the $i$-th marginal of the measure $\gamma$ on the
product space $X \times X$, i.e., $\proj_1\gamma(A) := \gamma(A \times X)$ and
$\proj_2\gamma(B) := \gamma(X \times B)$ whenever $A, B \subset X$.

Now, let $\Lip(X,\R^d)$ be the space of Lipschitz continuous functions on $X$
with values in $\R^d$ and $\Lip(X) := \Lip(X,\R^1)$.
Furthermore, denote the Lipschitz seminorm by $[\cdot]_{\Lip}$ so that
$[f]_{\Lip}$ is the Lipschitz constant of $f$.
Note that, if we fix some arbitrary $x_0 \in X$, the seminorm $[\cdot]_{\Lip}$
is actually a \emph{norm} on the set
\begin{equation}\label{eq:lip0}
    \Lip_0(X,\R^d) := \{ p \in \Lip(X,\R^d)\colon p(x_0) = 0 \}.
\end{equation}
The famous Kantorovich-Rubinstein duality \cite{kant1957} states that, for
$q=1$, the Wasserstein \emph{metric} is actually induced by a \emph{norm,} namely
$W_1(\mu, \mu') = \|\mu - \mu'\|_{\KR}$, where
\begin{equation}\label{eq:krdual}
    \|\nu\|_{\KR} := \sup\left\{
        \int_{X} p \dd \nu :
        ~p \in \Lip_0(X), ~[p]_{\Lip} \leq 1
    \right\},
\end{equation}
whenever $\nu \in \IM_0(X) := \{ \mu \in \IM\colon \int_X d\mu = 0\}$.
The completion $\KR(X)$ of $\IM_0(X)$ with respect to $\|\cdot\|_{\KR}$ is
a predual space of $(\Lip_0(X), [\cdot]_{\Lip})$
\cite[Thm. 2.2.2 and Cor. 2.3.5]{Weaver1999}.%
\footnote{The normed space $(\IM_0(X), \|\cdot\|_{\KR})$ is not complete unless
$X$ is a finite set \cite[Prop. 2.3.2]{Weaver1999}.
Instead, the completion of $(\IM_0(X), \|\cdot\|_{\KR})$ that we denote here
by $\KR(X)$ is isometrically isomorphic to the Arens-Eells space $AE(X)$.}
Hence, after subtracting a point mass at $x_0$, the set $\IP(X) - \delta_{x_0}$
is a subset of the Banach space $\KR(X)$, the predual of $\Lip_0(X)$.

Consequently, the embeddings
\begin{align}
    \IP(X) &\hookrightarrow (\KR(X), \|\cdot\|_{\KR}), \\
    \IP(X) &\hookrightarrow (\IM(X), \|\cdot\|_{\IM})
\end{align}
define two different topologies on $\IP(X)$.
The first embedding space $(\IM(X), \|\cdot\|_{\IM})$ is isometrically
isomorphic to the dual of $C(X)$.
The second embedding space $(\KR(X), \|\cdot\|_{\KR})$ is known to be a
metrization of the weak*-topology on the bounded subset $\IP(X)$ of the dual
space $\IM(X) = C(X)^*$~\cite[Thm. 6.9]{villani2009}.

Importantly, while $(\IP(X), \|\cdot\|_{\IM})$ is not separable unless $X$ is
discrete, $(\IP(X), \|\cdot\|_{\KR})$ is in fact compact, in particular complete
and separable \cite[Thm. 6.18]{villani2009} which is crucial in our result on
the existence of minimizers (Theorem~\ref{thm:existence}).
\section{Proof of $\TV$-Behavior for Cartoon-Like Functions}\label{apdx:cartoon-tv}
\begin{proof}[Prop.~\ref{prop:cartoon-tv}]
Let $p\colon \Omega \to (V^*)^d$ satisfy the constraints in \eqref{eq:bv-valued-tv}
and denote by $\nu$ the outer unit normal of $\partial U$.
The set $\Omega$ is bounded, $p$ and its derivatives are continuous and
$u \in L_w^\infty(\Omega, V)$ since the range of $u$ is finite and $U$, $\Omega$
are measurable.
Therefore all of the following integrals converge absolutely.
Due to linearity of the divergence,
\begin{align}
    &\langle \div p(x), u^\pm \rangle = \div (\langle p(\cdot), u^\pm \rangle), \\
    &\langle p(x), u^\pm \rangle := (
        \langle p_1(x), u^\pm \rangle, \dots, \langle p_d(x), u^\pm \rangle
    ) \in \R^d.
\end{align}
Using this property and applying Gauss' theorem, we compute
\begin{equation}
\begin{aligned}
    &\hspace{1.3em} \int_\Omega
        \langle -\div p(x), u(x) \rangle
    \dd x \\
    &= -\int_{\Omega \setminus U}
        \div (\langle p(x), u^- \rangle)
    \dd x
    - \int_{U}
        \div (\langle p(x), u^+ \rangle)
    \dd x \\
    &\hspace{-6.2pt}\overset{\text{Gauss}}{=} \int_{\partial U}
        \sum_{i=1}^d \langle \nu_i(x) p_i(x), {u^+} - {u^-} \rangle
    \dd \IH^{d-1}(x) \\
    &\leq \IH^{d-1}(\partial U) \cdot \|{u^+} - {u^-}\|_V.
\end{aligned}
\end{equation}
For the last inequality, we used our first assumption on $\|\cdot\|_{(V^*)^d}$
together with the norm constraint for $p$ in \eqref{eq:bv-valued-tv}.
Taking the supremum over $p$ as in \eqref{eq:bv-valued-tv}, we arrive at
\begin{equation}
    \TV_{V}(u) \leq \IH^{d-1}(\partial U) \cdot \|{u^+} - {u^-}\|_V.
\end{equation}

For the reverse inequality, let $\tilde p \in V^*$ be arbitrary with the
property $\|\tilde p\|_{V^*} \leq 1$
and $\phi \in C_c^1(\Omega, \R^d)$ satisfying $\|\phi(x)\|_2 \leq 1$.
Now, by \eqref{eq:product-norm-upper-bound}, the function
\begin{equation}
    p(x) := (\phi_1(x) \tilde p, \dots, \phi_d(x) \tilde p) \in (V^*)^d
\end{equation}
has the properties required in \eqref{eq:bv-valued-tv}.
Hence,
\begin{align}
    \TV_{V}(u) &\geq \int_\Omega
        \langle -\div p(x), u(x) \rangle
    \dd x \\
    &= -\int_\Omega \div \phi(x) \dd x
        \cdot \langle \tilde p, u^+ - u^- \rangle.
\end{align}
Taking the supremum over all $\phi \in C_c^1(\Omega, \R^d)$ satisfying
$\|\phi(x)\|_2 \leq 1$, we obtain
\begin{align}
    \TV_{V}(u) &\geq \Per(U, \Omega) \cdot \langle \tilde p, u^+ - u^- \rangle,
\end{align}
where $\Per(U, \Omega)$ is the perimeter of $U$ in $\Omega$.
In the theory of \emph{Caccioppoli sets} (or \emph{sets of finite perimeter}),
the perimeter is known to agree with $\IH^{d-1}(\partial U)$ for sets with $C^1$
boundary~\cite[p.~143]{ambrosio2000}.

Now, taking the supremum over all $\tilde p \in V^*$ with
$\|\tilde p\|_{V^*} \leq 1$ and using the fact that the canonical embedding of
a Banach space into its bidual is isometric, i.e.,
\begin{equation}
    \|u\|_V = \sup_{\|p\|_{V^*} \leq 1} \langle p, u \rangle,
\end{equation}
we arrive at the desired reverse inequality which concludes the proof.
\qed\end{proof}
\section{Proof of Rotational Invariance}\label{apdx:rot-invariance}
\begin{proof}[Prop.~\ref{prop:rot-invariance}]
    Let $R \in SO(d)$ and define
    \begin{align}
        R^T\Omega := \{ R^T x : x \in \Omega \},
        ~\tilde p(y) := R^T p(Ry).
    \end{align}
    In \eqref{eq:bv-valued-tv}, the norm constraint on $p(x)$ is
    equivalent to the norm constraint on $\tilde p(y)$ by condition
    \eqref{eq:product-norm-rot-inv}.
    Now, consider the integral transform
    \begin{align}
         \int_\Omega \langle -\div p(x), u(x) \rangle \dd x
         &= \int_{R^T\Omega} \langle -\div p(R y), \tilde u(y) \rangle \dd y \\
         &= \int_{R^T\Omega} \langle -\div \tilde p(y), \tilde u(y) \rangle \dd y.
    \end{align}
    where, using $R^T R = I$,
    \begin{align}
        \div \tilde p(y)
        &= \sum_{i=1}^d \partial_i \tilde p_i(y)
        = \sum_{i=1}^d \sum_{j=1}^d R_{ji} \partial_i \left[p_j(R y)\right] \\
        &= \sum_{i=1}^d \sum_{j=1}^d \sum_{k=1}^d
            R_{ji} R_{ki} \partial_k p_j(R y) \\
        &= \sum_{j=1}^d \sum_{k=1}^d \partial_k p_j(R y) \sum_{i=1}^d
            R_{ji} R_{ki} \\
        &= \sum_{j=1}^d \partial_j p_j(R y) = \div p(R y),
    \end{align}
    which implies $\TV_V(u) = \TV_V(\tilde u)$.
\qed
\end{proof}
\section{Discussion of Product Norms}\label{apdx:product-norms}
There is one subtlety about formulation \eqref{eq:bv-valued-tv} of the total
variation:
The choice of norm for the product space $(V^*)^d$ affects the properties of our
total variation seminorm.
\subsection{Product Norms as Required in Prop.~\ref{prop:cartoon-tv}}
The following proposition gives some examples for norms that satisfy or fail
to satisfy the conditions \eqref{eq:product-norm-lower-bound} and
\eqref{eq:product-norm-upper-bound} in Prop.~\ref{prop:cartoon-tv} about
cartoon-like functions.

\begin{proposition}\label{prop:product-norms}
    The following norms for $p \in (V^*)^d$ satisfy
    \eqref{eq:product-norm-lower-bound} and \eqref{eq:product-norm-upper-bound}
    for any normed space $V$:
    \begin{enumerate}
    \item For $s = 2$:
        \begin{equation}\label{eq:product-norm}
            \|p\|_{(V^*)^d,s} := \left(\sum_{i=1}^d \|p_i\|_{V^*}^s \right)^{1/s}.
        \end{equation}
    \item Writing $p(v):=(\langle p_1,v\rangle,\dots,\langle p_d,v\rangle)\in\R^d$,
        $v \in V$,
        \begin{equation}\label{eq:operator-norm}
            \|p\|_{\IL(V, \R^d)} := \sup_{\|v\|_V \leq 1} \|p(v)\|_{2}
        \end{equation}
    \end{enumerate}
    On the other hand, for any $1 \leq s < 2$ and $s > 2$, there is a normed
    space $V$ such that at least one of the properties
    \eqref{eq:product-norm-lower-bound}, \eqref{eq:product-norm-upper-bound} is
    not satisfied by the corresponding product norm \eqref{eq:product-norm}.
\end{proposition}
\begin{remark}
    In the finite-dimensional Euclidean case $V = \R^n$ with norm $\|\cdot\|_2$,
    we have $(V^*)^d = \R^{d,n}$, thus $p$ is matrix-valued and
    $\|\cdot\|_{\IL(V, \R^d)}$ agrees with the spectral norm
    $\|\cdot\|_\sigma$.
    The norm defined in \eqref{eq:product-norm} is the Frobenius norm
    $\|\cdot\|_F$ for $s=2$.
\end{remark}
\begin{proof}[Prop.~\ref{prop:product-norms}]
    By Cauchy-Schwarz,
    \begin{align}
        \left|\textstyle{\sum_{i=1}^d} x_i \langle p_i, v \rangle\right|
        &\leq \|x\|_2 \left(
                \textstyle{\sum_{i=1}^d} \left|\langle p_i, v \rangle\right|^2
            \right)^{1/2} \\
        &\leq \|x\|_2 \left(
                \textstyle{\sum_{i=1}^d} \|p_i\|_{V^*}^2 \|v\|_V^2
            \right)^{1/2} \\
        &\leq \|x\|_2 \|v\|_V \left(
                \textstyle{\sum_{i=1}^d} \|p_i\|_{V^*}^2
            \right)^{1/2},
    \end{align}
    whenever $p \in (V^*)^d$, $v \in V$, and $x \in \R^d$.
    Similarly, for each $q \in V^*$,
    \begin{equation}
        \left(
            \textstyle{\sum_{i=1}^d} \|x_i q\|_{V^*}^2
        \right)^{1/2}
        = \|x\|_2 \|q\|_{V^*}.
    \end{equation}
    Hence, for $s = 2$, the properties \eqref{eq:product-norm-lower-bound} and
    \eqref{eq:product-norm-upper-bound} are satisfied by the product norm
    \eqref{eq:product-norm}.

    For the operator norm \eqref{eq:operator-norm}, consider
    \begin{align}
        \left|\textstyle{\sum_{i=1}^d} x_i \langle p_i, v \rangle\right|
        &\leq \|x\|_2 \left(
                \textstyle{\sum_{i=1}^d} \left|\langle p_i, v \rangle\right|^2
            \right)^{1/2} \\
        &= \|x\|_2 \|p(v)\|_2 \\
        &\leq \|x\|_2 \|p\|_{\IL(V, \R^d)} \|v\|_V,
    \end{align}
    which is property \eqref{eq:product-norm-lower-bound}.
    On the other hand, \eqref{eq:product-norm-upper-bound} follows from
    \begin{align}
        \|(x_1 q, \dots, x_d q)\|_{\IL(V, \R^d)}
        &= \sup_{\|v\|_V \leq 1} \left(
            \textstyle{\sum_{i=1}^d} |x_i q(v)|^2
        \right)^{1/2} \\
        &= \|x\|_2 \sup_{\|v\|_V \leq 1} |q(v)| \\
        &= \|x\|_2 \|q\|_{V^*}.
    \end{align}

    Now, for $s > 2$, property \eqref{eq:product-norm-lower-bound} fails for
    $d = 2$, $V = V^* = \R$, $p = x = (1,1)$ and $v = 1$ since
    \begin{align}
        \left|\sum_{i=1}^d x_i \langle p_i, v \rangle\right|
        &= 2 > 2^{1/2} \cdot 2^{1/s} = \|x\|_2 \|p\|_{(V^*)^d,s} \|v\|_V.
    \end{align}
    For $1 \leq s < 2$, consider $d = 2$, $V^* = \R$, $q = 1$ and $x = (1,1)$,
    then
    \begin{align}
        \|(x_1 q, \dots, x_d q)\|_{(V^*)^d,s}
        &= 2^{1/s} > 2^{1/2} = \|x\|_2 \|q\|_{V^*},
    \end{align}
    which contradicts property \eqref{eq:product-norm-upper-bound}.
\qed
\end{proof}
\subsection{Rotationally Symmetric Product Norms}
For $V = (\R^n, \|\cdot\|_2)$, property \eqref{eq:product-norm-rot-inv}
in Prop.~\ref{prop:rot-invariance} is satisfied by the Frobenius norm as well
as the spectral norms on $(V^*)^d = \R^{d,n}$.
In general, the following proposition holds:
\begin{proposition}
    For any normed space $V$, the rotational invariance
    property~\eqref{eq:product-norm-rot-inv} is satisfied by the operator norm
    \eqref{eq:operator-norm}.
    For any $s \in [1,\infty)$, there is a normed space $V$ such that
    property \eqref{eq:product-norm-rot-inv} does not hold for the product norm
    \eqref{eq:product-norm}.
\end{proposition}
\begin{proof}
    By definition of the operator norm and rotational invariance of the
    Euclidean norm $\| \cdot \|_2$,
    \begin{align}
        \|Rp\|_{\IL(V, \R^d)} &= \sup_{\|v\|_V \leq 1} \|Rp(v)\|_{2} \\
        &= \sup_{\|v\|_V \leq 1} \|p(v)\|_{2} = \|p\|_{\IL(V, \R^d)}.
    \end{align}

    For the product norms \eqref{eq:product-norm}, without loss of generality,
    we consider the case $d = 2$, $V := (\R^2, \|\cdot\|_1)$, $p_1 = (1,0)$,
    $p_2 = (0,1)$ and
    \begin{equation}
        R := \begin{pmatrix}
            1/2 & -\sqrt{3}/2 \\
            \sqrt{3}/2 & 1/2
        \end{pmatrix} \in SO(2).
    \end{equation}
    Then $V^* := (\R^2, \|\cdot\|_\infty)$ and
    \begin{equation}
        \|p\|_{(V^*)^d,s} = \left(
            \textstyle{\sum_{i=1}^2} \|p_i\|_\infty^s
            \right)^{1/s}
        = 2^{1/s}
    \end{equation}
    whereas
    \begin{equation}
        (Rp)_1 = (1/2, -\sqrt{3}/2), ~(Rp)_2 = (\sqrt{3}/2, 1/2),
    \end{equation}
    \begin{align}
        \|Rp\|_{(V^*)^d,s} &= \left(
            \textstyle{\sum_{i=1}^2} (\sqrt{3}/2)^s
        \right)^{1/s} \\
        &= 2^{1/s} \cdot \sqrt{3}/2 \neq 2^{1/s} = \|p\|_{(V^*)^d,s},
    \end{align}
    for any $1 \leq s < \infty$.
\qed
\end{proof}
\subsection{Product Norms on $\Lip_0(X)$}\label{apdx:lip-product-norm}
We conclude our discussion about product norms on $(V^*)^d$ with the special case
of $V = \KR(X)$:
For $p \in [\Lip_0(X)]^d$, the most natural choice is
\begin{equation}\label{eq:lip-norm-Rd}
    [p]_{\Lip(X,\R^d)} := \sup_{z \neq z'} \frac{\|p(z) - p(z')\|^2_2}{d(z,z')},
\end{equation}
which is automatically rotationally invariant.
On the other hand, the product norm defined in \eqref{eq:product-norm} (with
$s=2$), namely $\sqrt{\sum_{i=1}^d [p_i]_{\Lip}^2}$, is not rotationally
invariant for general metric spaces $X$.
However, in the special case $X \subset (\R^n, \|\cdot\|_2)$ and
$p \in C^1(X,\R^d)$, the norms \eqref{eq:lip-norm-Rd} and \eqref{eq:product-norm}
coincide with $\sup_{z\in X} \|Dp(z)\|_\sigma$ (spectral norm of the Jacobian)
and $\sup_{z\in X} \|Dp(z)\|_F$ (Frobenius norm of the Jacobian) respectively,
both satisfying rotational invariance.
\section{Proof of Non-Uniqueness}\label{apdx:non-uniqueness}
\begin{proof}[Prop.~\ref{prop:non-uniqueness}]
    Let $u \in L_w^\infty(\Omega, \IP(X))$.
    With the given choice of $X$, there exists a measurable function
    $\tilde u\colon \Omega \to [0,1]$ such that
    \begin{equation}
        u(x) = \tilde u(x) \delta_0 + (1 - \tilde u(x)) \delta_1.
    \end{equation}
    The measurability of $\tilde u$ is equivalent to the weak measurability of
    $u$ by definition:
    \begin{align}
        \langle p, u(x) \rangle
        &= \tilde u(x) \cdot p_0 + (1 - \tilde u(x)) \cdot p_1 \\
        &= \tilde u(x) \cdot (p_0 - p_1) + p_1.
    \end{align}
    The constraint
    \begin{equation}
        p \in C_c^1(\Omega, [\Lip_0(X)]^d), ~[p(x)]_{\Lip(X,\R^d)} \leq 1
    \end{equation}
    from the definition of $\TV_{\KR}$ in \eqref{eq:measure-valued-tv}
    translates to
    \begin{equation}
        p_0, p_1 \in C_c(\Omega,\R^d),~\|p_0(x) - p_1(x)\|_2 \leq 1.
    \end{equation}
    Furthermore,
    \begin{align}
        &\langle -\div p(x), u(x) \rangle \\
        &= -\div p_0(x) \cdot \tilde u(x) - \div p_1(x) \cdot (1-\tilde u(x)) \\
        &= -\div (p_0 - p_1)(x) \cdot \tilde u(x) - \div p_1(x).
    \end{align}
    By the compact support of $p_1$, the last term vanishes when integrated over
    $\Omega$.
    Consequently,
    \begin{align}
        \TV_{\KR}(u)
        &= \sup\left\{
            \int_\Omega -\div (p_0 - p_1) (x) \cdot \tilde u(x) \dd x : \right.\\
        &\left.\phantom{\sup\{\int_\Omega}\hspace{0.7cm}
            p_0, p_1 \in C_c(\Omega,\R^d),~\|(p_0 - p_1)(x)\|_2 \leq 1
        \right\} \\
        &= \sup\left\{ \int_\Omega -\div p(x) \cdot \tilde u(x) \dd x : \right.\\
        &\quad\left.\phantom{\sup\{\int_\Omega}\hspace{2cm}
            p \in C_c(\Omega,\R^{d}),~\|p(x)\|_2 \leq 1
        \right\} \\
        &= \TV(\tilde u).
    \end{align}
    and therefore
    \begin{align}
        T_{\rho,\lambda}(u)
        &= \int_{\Omega \setminus U} \hspace{-8pt}\tilde u(x) \dd x
           + \int_U (1-\tilde u(x)) \dd x
           + \lambda\TV(\tilde u) \\
        &= \int_\Omega |\mathbf{1}_U(x) - \tilde u(x)| \dd x
           + \lambda \TV(\tilde u) \\
        &= \|\mathbf{1}_U - \tilde u\|_{L^1} + \lambda\TV(\tilde u).
    \end{align}
    Thus we have shown that the functional $T_{\rho,\lambda}$ is equivalent to
    the classical $L^1$-$\TV$ functional with the indicator function
    $\mathbf{1}_U$ as input data and evaluated at $\tilde u$ which is known to
    have non-unique minimizers for a certain choice of $\lambda$ \cite{Chan2005}.
\qed
\end{proof}
\section{Proof of Existence}\label{apdx:existence}
\subsection{Well-Defined Energy Funtional}
In order for the functional defined in \eqref{eq:T-definition} to be
well-defined, the mapping $x \mapsto \rho(x, u(x))$ needs to be measurable.
In the following Lemma, we show that this is the case under mild conditions on
$\rho$.
\begin{lemma}\label{lem:rho-measurable}
    Let $\rho\colon \Omega \times \IP(X) \to [0,\infty)$ be a globally bounded
    function that is measurable in the first and convex in the second variable,
    i.e., $x \mapsto \rho(x,\mu)$ is measurable for each $\mu \in \IP(X)$, and
    $\mu \mapsto \rho(x,\mu)$ is convex for each $x \in \Omega$.
    Then the map $x \to \rho(x,u(x))$ is measurable for every
    $u \in L_w^\infty(\Omega, \IP(X))$.
\end{lemma}
\begin{remark}
    As will become clear from the proof, the convexity condition can be replaced
    by the assumption that $\rho$ be continuous with respect to $(\IP(X), W_1)$
    in the second variable.
    However, in order to ensure weak* lower semi-continuity of the functional
    \eqref{eq:T-definition}, we will require convexity of $\rho$ in the
    existence proof (Thm.~\ref{thm:existence}) anyway.
    Therefore, for simplicity we also stick to the (stronger) convexity
    condition in Lemma~\ref{lem:rho-measurable}.
\end{remark}
\begin{remark}
    One example of a function satisfying the assumptions in
    Lemma~\ref{lem:rho-measurable} is given by
    \begin{equation}
        \rho(x,\mu) := W_1(f(x),\mu), ~x \in \Omega, ~\mu \in \IP(\IS^2).
    \end{equation}
    Indeed, boundedness follows from the boundedness of the Wasserstein metric
    in the case of an underlying bounded metric spaces (here $\IS^2$).
    Convexity in the second argument follows from the fact that the Wasserstein
    metric is induced by a norm \eqref{eq:krdual}.
\end{remark}
\begin{proof}[Lemma~\ref{lem:rho-measurable}]
    The metric space $(\IP(X), W_1)$ is compact, hence separable.
    By Pettis' measurability theorem
    \cite[Chapter VI, §1, No. 5, Prop. 12]{bourbaki2004}, weak and strong
    measurability coincide for separably-valued functions, so that
    $u$ is actually strongly measurable as a function with values in
    $(\IP(X),W_1)$.
    Note, however, that this does not imply strong measurability
    with respect to the norm topology of $(\IM(X), \|\cdot\|_{\IM})$ in general!

    As bounded convex functions are locally Lipschitz continuous
    \cite[Thm. 2.34]{clarke2013}, $\rho$ is continuous in the second variable
    with respect to $W_1$.
    As in the proof of Lemma \ref{lem:measurability1}, we now note that $\rho$
    is a Carathéodory function, for which compositions with measurable functions
    such as $x \mapsto \rho(x,u(x))$ are known to be measurable.
\qed
\end{proof}
\subsection{The Notion of Weakly*~Measurable Functions}
Before we can go on with the proof of existence of minimizers to
\eqref{eq:T-definition}, we introduce the notion of \emph{weak*~measurability}
because this will play a crucial role in the proof.

Analogously with the notion of weak~measurability and with
$L_{w}^\infty(\Omega, \KR(X))$ introduced above, we say that a measure-valued
function $u\colon \Omega \to \IM(X)$ is \emph{weakly*~measurable} if the mapping
\begin{equation}
    x \mapsto \int_X f(z) \dd u_x(z)\label{eq:weakstarmap}
\end{equation}
is measurable for each $f \in C(X)$.
$L_{w*}^\infty(\Omega, \IM(X))$ is defined accordingly as the space of
weakly*~measurable functions.

For functions $u\colon \Omega \to \IP(X)$ mapping onto the space of probability
measures, there is an immediate connection between weak* measurability and weak
measurability: $u$ is weakly~measurable if the mapping
\begin{equation}
x \mapsto \int_X p(z) \dd u_x(z)
\end{equation}
is measurable whenever $p \in \Lip_0(X)$.
However, since,  by the Stone-Weierstrass theorem, the Lipschitz functions
$\Lip(X)$ are dense in $(C(X), \|\cdot\|_{\infty})$~\cite[p.~198]{carothers2000},
both notions of measurability coincide for probability measure-valued functions
$u\colon \Omega \to \IP(X)$, so that
\begin{equation}
    L_w^\infty(\Omega, \IP(X)) = L_{w*}^\infty(\Omega, \IP(X)).
\end{equation}

However, as this equivalence does not hold for the larger spaces
$L_{w*}^\infty(\Omega, \IM(X))$ and $L_{w}^\infty(\Omega, \IM(X))$, it will be
crucial to keep track of the difference between weak and weak* measurability in
the existence proof.
\subsection{Proof of Existence}
\begin{proof}[Theorem~\ref{thm:existence}]
The proof is guided by the direct method from the calculus of variations.
The first part is inspired by the proof of the Fundamental Theorem for Young
measures as formulated and proven in~\cite{Ball1989}.

Let $u^k\colon \Omega \to \IP(X)$, $k \in \IN$, be a minimizing
sequence for $T_{\rho,\lambda}$, i.e.,
\begin{equation}
    T_{\rho,\lambda}(u^k) \to \inf_{u} T_{\rho,\lambda}(u)
        ~\text{ as }~
    k \to \infty.
\end{equation}
As $\IM(X)$ is the dual space of $C(X)$, $L_{w*}^\infty(\Omega, \IM(X))$ with
the norm defined in \eqref{eq:l-inf-w-norm} is dual to the Banach space
$L^1(\Omega, C(X))$ of Bochner integrable functions on $\Omega$ with values in
$C(X)$ \cite[p. 93]{IonescuTulcea1969}.
Now, $\IP(X)$ as a subset of $\IM(X)$ is bounded so that our sequence $u^k$ is
bounded in $L_{w*}^\infty(\Omega, \IM(X))$ (here we use again that
$L_{w*}^\infty(\Omega, \IP(X)) = L_{w}^\infty(\Omega, \IP(X))$).

Note that we get boundedness of our minimizing sequence ``for free'', without any
assumptions on the coercivity of $T_{\rho,\lambda}$!
Hence we can apply the Banach-Alaoglu theorem, which states that there exist
$u^\infty \in L_{w*}^\infty(\Omega, \IM(X))$ and a subsequence, also denoted by
$u^k$, such that
\begin{equation}
    u^k \overset{\ast}{\rightharpoonup} u^\infty
    \text{ in } L_{w*}^\infty(\Omega, \IM(X)).
\end{equation}
Using the notation in \eqref{eq:dual-pairing}, this means by definition
\begin{align}\label{eq:weak-star-conv}
    \int_\Omega \langle u^k(x), p(x) \rangle \dd x
    \to \int_\Omega \langle u^\infty(x), \,p(x) \rangle \dd x
    \quad\forall p \in L^1(\Omega, C(X)).
\end{align}

We now show that $u^\infty(x) \in \IP(X)$ almost everywhere, i.e., $u^\infty$
is a nonnegative measure of unit mass:
The convergence \eqref{eq:weak-star-conv} holds in particular for the choice
$p(x,s) := \phi(x)f(s)$, where $\phi \in L^1(\Omega)$ and $f \in C(X)$.
For nonnegative functions $\phi$ and $f$, we have
\begin{equation}
    \int_\Omega \phi(x) \langle u^k(x), f \rangle \dd x \geq 0
\end{equation}
for all $k$, which implies
\begin{equation}
    \int_\Omega \phi(x) \langle u^\infty(x), f \rangle \dd x \geq 0.
\end{equation}
Since this holds for all nonnegative $\phi$ and $f$, we deduce that $u^\infty(x)$
is a nonnegative measure for almost every $x \in \Omega$.
The choice $f(s) \equiv 1$ in \eqref{eq:weak-star-conv} shows that
$u^\infty$ has unit mass almost everywhere.

Therefore $u^\infty(x) \in \IP(X)$ almost everywhere and we have shown that
$u^\infty$ lies in the feasible set $L_{w}^\infty(\Omega, \IP(X))$.
It remains to show that $u^\infty$ is in fact a minimizer.

In order to do so, we prove weak* lower semi-con\-tinuity of $T_{\rho,\lambda}$.
We consider the two integral terms in the definition
\eqref{eq:T-definition} of $T_{\rho,\lambda}$ separately.
For the $\TV_{\KR}$ term, for any $p \in C_c^1(\Omega, \Lip(X,\R^d))$, we have
$\div p \in L^1(\Omega, C(X))$ so that
\begin{equation}
    \lim_{k\to\infty} \int_\Omega
        \langle \div u^k(x), p(x) \rangle
    \dd x
    = \int_\Omega
        \langle \div u^\infty(x), p(x) \rangle
    \dd x.
\end{equation}
Taking the supremum over all $p$ with $[p(x)]_{[\Lip(X)]^d} \leq 1$ almost
everywhere, we deduce lower semi-continuity of the regularizer:
\begin{equation}
    \TV_{\KR}(u^\infty) \leq \liminf_{k\to\infty} \TV_{\KR}(u^k).
\end{equation}
The data fidelity term $u \mapsto \int_\Omega \rho(x,u(x)) \dd x$ is convex and
bounded on the closed convex subset $L_w^\infty(\Omega, \IP(X))$ of the space
$L_{w*}^\infty(\Omega, \IM(X))$.
It is also continuous, as convex and bounded functions on normed spaces are
locally Lipschitz-continuous.
This implies weak* lower semi-continuity on $L_w^\infty(\Omega, \IP(X))$.

Therefore, the objective function $T_{\rho,\lambda}$ is weakly* lower
semi-continuous, and we obtain
\begin{equation}
    T_{\rho,\lambda}(u^\infty)
        \leq \lim\inf_{k\to\infty} T_{\rho,\lambda}(u^k)
\end{equation}
for the minimizing sequence $(u^k)$, which concludes the proof.
\qed
\end{proof}

\bibliographystyle{spmpsci}
\bibliography{paper.bib}

\end{document}